\documentclass[11pt]{article}
\usepackage{amsfonts}
\usepackage{mathrsfs}
\usepackage{amsthm}
\usepackage{amsmath}
\usepackage{amssymb}
\usepackage{latexsym}
\usepackage{graphicx}
\usepackage{color,varioref}
\usepackage{longtable}
\usepackage{rotating,multirow,array}
\usepackage{enumerate}
\usepackage{caption}
\usepackage{float}
\definecolor{red}{rgb}{0.9,0,0}
\definecolor{green}{rgb}{0,0.9,0}
\definecolor{blue}{rgb}{0,0,0.9}
\definecolor{pink}{rgb}{0.8,0,0.4}

\def\mc{\multicolumn}

 \def\R{{\cal R}}

\def\norm#1{\|#1\|}

\def\lam{{\lambda}}

\def\ni{\noindent}

\def\R{\mathbb{R}}
\def\bone{{\bf 1}}
\def\betaSc{\beta^{\bar{S}}}
\def\truebetaSc{\ddot{\beta}^{\bar{S}}}
\def\truebeta{\ddot{\beta}}
\def\Sc{\bar{S}}
\def\pstar{p_*}
\def\pone{p}
\def\ptwo{q}
\def\tauone{\sigma}
\def\tautwo{\tau}
\def\norm#1{\|#1\|}

\newtheorem{theorem}{Theorem}[section]
\newtheorem{proposition}{Proposition}[section]
\newtheorem{lemma}{Lemma}[section]

\newtheorem{remark}{Remark}[section]
\newtheorem{definition}{Definition}[section]
\newtheorem{assumption}{Assumption}[section]

\begin{document}

\title{\bf  A sparse semismooth Newton based proximal majorization-minimization algorithm for nonconvex square-root-loss regression problems }

\author{Peipei Tang\thanks{School of Computer and
Computing Science, Zhejiang University City College, Hangzhou
310015, China (tangpp@zucc.edu.cn). This author's research is supported by the Natural Science Foundation of Zhejiang Province of China under Grant No. LY19A010028 and the Science $\&$ Technology Development Project of Hangzhou, China under Grant No. 20170533B22.}, Chengjing Wang\thanks{{\bf Corresponding author}, School of Mathematics, Southwest Jiaotong University, No. 999, Xian Road, West Park, High-tech Zone, Chengdu 611756, China  ({\tt renascencewang@hotmail.com}).}, Defeng
Sun\thanks{Department of Applied Mathematics, The Hong Kong
Polytechnic University, Hung Hom, Hong Kong  (defeng.sun@polyu.edu.hk). {This author is supported by Hong Kong Research Grant Council grant PolyU153014/18p and Shenzhen Research Institute
of Big Data, Shenzhen 518000 grant 2019ORF01002.}}, and Kim-Chuan Toh\thanks{Department of Mathematics, National
University of Singapore (mattohkc@nus.edu.sg).
This author's research is supported in part by
the  Academic Research Fund of the Ministry of Education of Singapore under Grant No. R-146-000-257-112.
Part of this research is done while the author is visiting the Shenzhen Research Institute
of Big Data at the Chinese University of Hong Kong at Shenzhen.
}}

\maketitle

\begin{abstract}
In this paper, we consider high-dimensional nonconvex square-root-loss regression problems and introduce a proximal majorization-minimization (PMM) algorithm for solving these problems. Our key idea for making the proposed PMM to be efficient is to develop a sparse semismooth Newton method to solve the corresponding subproblems.  By using the Kurdyka-{\L}ojasiewicz property exhibited in the underlining problems, we prove that the PMM algorithm converges to a d-stationary point. We also analyze the oracle property of the initial subproblem used in our algorithm. Extensive numerical experiments are presented to demonstrate the high efficiency of the proposed PMM algorithm.
\end{abstract}


\section{Introduction}\label{sec:Introduction}
Sparsity estimation is one of the most important problems in
statistics, machine learning and signal processing.
One typical example on this aspect is to
estimate a vector $\hat{\beta}$ from a high-dimensional linear
regression model
\begin{eqnarray*}
b=X\truebeta+\varepsilon,
\end{eqnarray*}
where $X\in \R^{m\times n}$ is the design matrix, $b\in \R^m$ is the response
vector,  and $\varepsilon\in \R^m$ is the noise vector for which each of its component
$\varepsilon_{i}$ has zero-mean and unknown variance
$\varsigma^{2}$. In many applications, the number of attributes $n$ is
much larger than the sample size $m$ and
$\truebeta$ is known to be sparse a priori. Under the assumption of sparsity, a
regularizer which controls the overfitting and/or variable selection is often added to the model. One of the most commonly used
regularizers in practice  is the $\ell_{1}$ norm and the resulting model,
first proposed  in \cite{Tibshirani1996}, is usually referred to as
the Lasso model, which is given by
\begin{eqnarray}
\min_{\beta\in\R^{n}}\left\{\frac{1}{2}\norm{X
\beta-b}^{2}+\lambda\norm{\beta}_{1}\right\},\label{eq:Lasso}
\end{eqnarray}
 where $\|\cdot\|$ is the Euclidean norm in $\R^m$.
 The Lasso estimator produced from (\ref{eq:Lasso})  is
computationally attractive because it minimizes a structured convex
function. Moreover, when the error vector $\varepsilon$
follows a normal distribution and suitable design
conditions hold, this estimator achieves a near-oracle
performance. Despite having these attractive features, the Lasso
recovery of $\truebeta$ relies on knowing the standard deviation of the noise.
However, it is non-trivial to estimate the deviation when the feature
dimension is large, particularly when it is much
larger than the sample size. To overcome the
aforementioned defect, the authors in \cite{BelloniCW2011} proposed a new
estimator that solves the square-root Lasso (srLasso) model
\begin{eqnarray}\label{eq:SRLasso}
\min_{\beta\in\R^{n}}\Big\{\norm{X
\beta-b}+\lambda\norm{\beta}_{1}\Big\},
\end{eqnarray}
which eliminates the need to know or to pre-estimate
the deviation. It has been shown (see e.g., \cite{Bellec,Derumigny}) that the srLasso estimator can achieve the minimax optimal rate of convergence under some suitable conditions, even though the noise level $\varsigma$ is unknown.
It is worth  mentioning that the scaled Lasso proposed by Sun and Zhang \cite{SunZ2012} is essentially equivalent to the srLasso model \eqref{eq:SRLasso}. The solution approach proposed by the authors in \cite{SunZ2012} is to iteratively solve a sequence of Lasso problems, which can be expensive numerically. Moreover, Xu, Caramanis and Mannor \cite{XuCM} proved  that the srLasso model \eqref{eq:SRLasso} is equivalent to a robust linear regression problem subject to an uncertainty set that bounds the norm of the disturbance to each feature, which itself is an ideal approach of reducing sensitivity of linear regression.

The Lasso problem and the srLasso problem are both convex and computationally attractive. A number of algorithms, such as the accelerated proximal gradient (APG) method \cite{BeckT2009},  interior-point method (IPM) \cite{KimKLBG2007}, and  least angle regression (LARS) \cite{EfronHJT} have been proposed to solve the Lasso problem \eqref{eq:Lasso}. In a very recent work \cite{LiSunToh2018}, Li, Sun and Toh proposed a highly efficient semismooth Newton augmented Lagrangian method to solve the Lasso problem (\ref{eq:Lasso}).
In contrast to the Lasso problem \eqref{eq:Lasso}, there are currently
no
efficient algorithms for solving the more challenging
srLasso problem \eqref{eq:SRLasso} due to the fact that the square-root loss function
in the objective is nonsmooth. Notably the alternating direction method of multipliers (ADMM) was applied to solve the srLasso problem \eqref{eq:SRLasso} by the authors   in \cite{Liuhan2015}. However, as can be seen from the numerical experiments conducted later in this paper, the ADMM  approach is not very efficient in solving  many large-scale problems.

Going beyond the $\ell_{1}$ norm regularizer, other regularization functions for variables selection are often used to avoid overfitting in the area of support vector machines and other statistical applications.
It has also been shown
that, instead of a convex relaxation with the  $\ell_{1}$ norm, a proper
nonconvex regularization can achieve a sparse estimation with fewer
measurements, 
and is more robust
against noises \cite{Chartrand2007,Chen2014}.
After the pioneering work of Fan and Li \cite{Fan2001}, various
nonconvex sparsity
functions have been proposed as surrogates of the $\ell_0$ function
in the
last decade and they have been used as regularizers to avoid model overfitting (see e.g. \cite{Hastie2015}) in high-dimensional statistical learning.  It has been proven that each of these
nonconvex surrogate sparsity
functions can be expressed as the difference of two convex functions \cite{Ahn2017,Le-thi2015}.
Given the 
d.c. (difference of convex functions) property of these
nonconvex regularizers, it is natural for one to design a majorization-minimization algorithm to
solve the nonconvex problem. Such an exploitation of the d.c. property of the
regularization function had  been considered in the majorized
penalty approach proposed by Gao and Sun \cite{GaoS2010} for
solving a rank constrained correlation matrix estimation problem.

In this paper, we aim to develop an efficient and robust algorithm
for solving the following square-root regression problem
\begin{eqnarray}\label{eq:prob}
\min_{\beta\in\R^{n}}\Big\{g(\beta):=\norm{X\beta-b}+
\lambda p(\beta)-q(\beta) \Big\},\label{eq:ENSRLasso}
\end{eqnarray}
where the first part of the regularization function
$p:\R^{n}\rightarrow \R_+$ is a norm function whose proximal mapping is strongly semismooth and the second part
$q:\R^{n}\rightarrow\R$ is a convex smooth
function (the dependence of $q$ on $\lambda$ has been dropped here). We should note that the assumption made on $p$ is 
rather
mild as the proximal mappings of many interesting functions, such as the $l_{1}$-norm function, are strongly semismooth \cite{Meng2005}.
For the case when $q\equiv 0$, the oracle property of the
model has been established in \cite{Sara2017} when $p$ is a weakly decomposable norm.
For the need of efficient computations, here we shall
extend the analysis to the same model but with
the proximal terms $\frac{\tauone}{2}\norm{\beta}^2 + \frac{\tautwo}{2}
\norm{X\beta-b}^2$ added, where $\tauone\geq 0$ and $\tautwo\geq 0$ are given
parameters.

Based on the {d.c.} structure of the nonconvex regularizer in \eqref{eq:ENSRLasso},
we    design a two stage proximal majorization-minimization (PMM)
algorithm to solve the  problem (\ref{eq:prob}).
In both stages of
 the PMM algorithm, the key step in each iteration is to solve a
convex subproblem whose objective contains the sum of two nonsmooth
functions, namely, the square-root-loss regression function and $p(\cdot)$.
One of the main contributions of this paper is in proposing a novel proximal majorization
approach to solve the said subproblem via its dual by a highly efficient semismooth
Newton method.
We  also analyze the convergence properties of our algorithm.
By using the Kurdyka-{\L}ojasiewicz (KL) property exhibited in the underlying problem, we prove that the PMM algorithm converges to a d-stationary point.
In the last part of the paper, we
present comprehensive numerical
results to demonstrate the efficiency of our semismooth Newton based PMM algorithm.
Based on the performance of our algorithm against two natural first-order methods, namely
the primal based and dual based ADMM (for the convex case), we can safely conclude that our algorithmic framework
is far superior for solving the square-root regression problem \eqref{eq:ENSRLasso}.

\section{Preliminary}
Let $f:\R^{n}\rightarrow (-\infty,+\infty]$ be a proper closed convex function.
The Fenchel conjugate of $f$ is defined by $f^{*}(x):=\sup_{y\in\R^{n}}
\{\langle y,x\rangle-f(y)\}$,
the proximal mapping and the Moreau envelope function of $f$ with parameter $t>0$  are defined, respectively, as
\begin{eqnarray*}
\mbox{Prox}_{f}(x)&:=&\mathop{\mbox{argmin}}_{y\in\R^{n}}\Big\{f(y)+\frac{1}{2}\|y-x\|^{2}\Big\},\quad\forall\, x\in\R^{n},\\
\Phi_{tf}(x)&:=&\min_{y\in\R^{n}}\Big\{f(y)+\frac{1}{2t}\|y-x\|^{2}\Big\},\quad\forall\, x\in\R^{n}.
\end{eqnarray*}
Let $t>0$ be a given parameter. Then by Moreau's identity theorem (see, e.g., Theorem 31.5 of \cite{Rockafellar96}), we have that
\begin{eqnarray}\label{Moreau-identity}
\mbox{Prox}_{tf}(x)+t\mbox{Prox}_{f^{*}/t}(x/t)=x, \quad\forall\; x \in \R^n.
\end{eqnarray}
We also know from Proposition 13.37 of \cite{Rockafellar} that $\Phi_{tf}$ is continuously differentiable with
\begin{eqnarray}\label{differentiable-Moreau-envelope}
\nabla \Phi_{tf}(x)=t^{-1}(x-\mbox{Prox}_{tf}(x)), \quad\forall\; x \in \R^n.
\end{eqnarray}

Given a set $C\subseteq\R^{n}$ and an arbitrary collection of functions $\left\{f_{i}\ |\ i\in I\right\}$ on $\R^{n}$, we denote $\delta_{C}(\cdot)$ as the indicator function of $C$ such that $\delta_{C}(\beta)=0$ if $\beta\in C$ and $\delta_{C}(\beta)=+\infty$ if $\beta\notin C$, $\mbox{conv}(C)$ as the convex hull of $C$ and $\mbox{conv}\{f_{i}\ |\ i\in I\}$ as the convex hull of the pointwise infimum of the collection. This means that $\mbox{conv}\{f_{i}\ |\ i\in I\}$ is the greatest convex function $f$ on $\R^{n}$ such that $f(\beta)\leq f_{i}(\beta)$ for any $\beta\in\R^{n}$ and $i\in I$.

Next we present some results in variational analysis from
\cite{Rockafellar}. Let $\Psi:\mathcal{O}\rightarrow\R^{m}$ be a locally Lipschitz continuous vector-valued
function
defined on an open set $\mathcal{O}\subseteq\R^{n}$. It follows from \cite[Theorem 9.60]{Rockafellar}
 that $\Psi$ is
F(r\'{e}chet)-differentiable almost everywhere on $\mathcal{O}$. Let
$\mathcal{D}_{\Psi}$ be the set of all points where $\Psi$ is
F-differentiable and $J\Psi(x)\in\R^{m\times n}$ be the
Jacobian of $\Psi$ at $x\in\mathcal{D}_{\Psi}$. For any $x\in\mathcal{O}$, the B-subdifferential of $\Psi$ at $x$ is defined
by
\begin{eqnarray*}
\partial_{B}\Psi(x):=\left\{V\in\R^{m\times n}\left|\exists\, \{x^{k}\}\subseteq\mathcal{D}_{\Psi}\ \mbox{such that}\ \lim_{k\rightarrow\infty}x^{k}=x\ \mbox{and}\ \lim_{k\rightarrow\infty}J\Psi(x^{k})=V\right.\right\}.
\end{eqnarray*}
The Clarke subdifferential of $\Psi$ at $x$ is defined as the convex
hull of the B-subdifferential of $\Psi$ at $x$, that is
$\partial_{C}\Psi(x):=\mbox{conv}(\partial_{B}\Psi(x))$.

Let   $\phi$ be defined from  $\R^n$ to $\R$. The Clarke subdifferential of $\phi$ at $x\in\R^{n}$ is defined by
\begin{eqnarray*}
\partial_{C}\phi(x):=\left\{h\in\R^{n}\left|\limsup_{x'\rightarrow x,t\downarrow 0}\frac{\phi(x'+ty)-\phi(x')-th^{T}y
}{t}\geq 0\right.,\ \forall\, y\in\R^{n}\right\}.
\end{eqnarray*}
The  regular subdifferential of $\phi$ at $x\in \R^n$ is defined as
\begin{eqnarray*}
\hat{\partial}\phi(x):=\left\{h\in\R^{n}\left|\liminf_{x\neq
y\rightarrow x}\frac{\phi(y)-\phi(x)-h^{T}(y-x)}{\|y-x\|}\geq
0\right.\right\}
\end{eqnarray*}
and the limiting subdifferential of $\phi$ at $x$ is defined as
\begin{eqnarray*}
\partial\phi(x):=\left\{h\in\R^{n}\left|\exists\ \{x^{k}\}\rightarrow x\ \mbox{and}\ \{h^{k}\}\rightarrow h\ \mbox{satisfying}\ h^{k}\in\hat{\partial}\phi(x^{k}),\ \forall\, k\right.\right\}.
\end{eqnarray*}
If $\phi$ is a convex function, then the Clarke subdifferential, the regular
 subdifferential and the limiting subdifferential of $\phi$ at $x$ coincide with the set of (transposed)  subgradients of $\phi$ at $x$ in the sense of convex analysis.

We know from Theorem 10.1 of
\cite{Rockafellar} that 
$0 \in \hat{\partial}\phi(\bar x)$
is a necessary condition for $\bar x\in \R^n$ to be a local minimizer of $\phi$.
If the
function $\phi$ (may not be convex) is locally Lipschitz continuous near $\bar x$ and directionally differentiable at $\bar x$, then $0 \in \hat{\partial}\phi(\bar x)$
 is equivalent to the
directional-stationarity (d-stationarity) of $\bar x$, that is
\begin{eqnarray*}
\phi'(\bar x;h):=\lim_{\lambda\downarrow 0}\frac{\phi(\bar x+\lambda
h)-\phi(\bar x)}{\lambda}\geq 0,\quad \forall\, h\in \R^n.
\end{eqnarray*}
In this paper, we will prove that the sequence generated by our algorithm converges to a d-stationary point of the problem.

For further discussions, we recall the concept of semismoothness
originated from   \cite{Mifflin1977,QiSun1993} and other two definitions  used in \cite{Sara2014}.
\begin{definition}
Let
$F:\mathcal{O}\subseteq\R^{n}\rightarrow\R^{m}$ be
a locally Lipschitz continuous function and
$\mathcal{K}:\mathcal{O}\rightrightarrows\R^{m\times n}$ be
a nonempty and compact valued, upper-semicontinuous set-valued
mapping on the open set $\mathcal{O}$.
$F$ is said to be semismooth
 at $v\in\mathcal{O}$ with respect to the set-valued
mapping
$\mathcal{K}$ if $F$ is directionally differentiable at $v$ and for
any $\Gamma\in\mathcal{K}(v+\Delta v)$ with $\Delta v\rightarrow0$,
\begin{eqnarray*}
F(v+\Delta v)-F(v)-\Gamma\Delta v=o(\|\Delta v\|).
\end{eqnarray*}
$F$ is said to be $\gamma$-order $(\gamma>0)$ (strongly, if
$\gamma=1$) semismooth at $v$ with respect to $\mathcal{K}$
if $F$ is semismooth at $v$ and for any
$\Gamma\in\mathcal{K}(v+\Delta v)$,
\begin{eqnarray*}
F(v+\Delta v)-F(v)-\Gamma\Delta v=O(\|\Delta v\|^{1+\gamma}).
\end{eqnarray*}
$F$ is called a semismooth ($\gamma$-order semismooth, strongly
semismooth) function on $\mathcal{O}$ with respect to $\mathcal{K}$
if it is semismooth ($\gamma$-order semismooth, strongly semismooth)
at every $v\in\mathcal{O}$ with respect to $\mathcal{K}$.
\end{definition}

\begin{definition}
The  norm function $p$ defined in $\R^{n}$ is said to be weakly decomposable
for an index set $S\subset\{1,2,\ldots,n\}$ if there exists a norm
$p^{\Sc}$ defined on $\R^{|\Sc|}$ such that
\begin{eqnarray*}
p(\beta)\geq
p(\beta_{S})+p^{\Sc}(\betaSc),\quad \forall\, \beta\in\R^{n},
\end{eqnarray*}
where $\Sc=\{1,2,\ldots,n\}\backslash S$, $\beta_{S} = \beta\circ \bone_S$
  and
$\betaSc := (\beta_j)_{j\in \Sc}\in \R^{|\Sc|}$. Here  $\bone_S\in \R^n$ denotes the indicator vector of $S$  and ``$\circ$'' denotes the elementwise product.
\end{definition}
The weakly decomposable property of a norm is a relaxation of the decomposability property of the $\ell_{1}$ norm. It has been proved in \cite{Sara2017} that many interesting regularizers such as the sparse group Lasso and SLOPE are weakly decomposable. A set $S$ is
said to be an allowed set if $p$ is a weakly decomposable norm for this set. We say
that a vector $\beta\in\R^{n}$ satisfies the $(L,S)$-cone condition for a norm $p$ if $p^{\Sc}(\betaSc)\leq L p(\beta_{S})$ with $L>0$ and $S$ being an allowed set.
\begin{definition}
Given $X\in\mathbb{R}^{m\times n}$. Let $S$ be an allowed set of a weakly decomposable norm $p$ and
$L>0$ be a constant. Then the $p$-eigenvalue is defined
as
\begin{eqnarray*}
\delta_{p}(L,S):=\min\left\{\|X\beta_{S}-X\beta_{\Sc}\|\, \left|\, p(\beta_{S})=1,\
p^{\Sc}(\betaSc)\leq L,\ \beta\in\mathbb{R}^{n}\right.\right\}.
\end{eqnarray*}
The $p$-effective sparsity is defined as
\begin{eqnarray*}
\Gamma_{p}(L,S):=\frac{1}{\delta_{p}(L,S)}.
\end{eqnarray*}
\end{definition}
Note that the $p$-eigenvalue defined above is a generalization of the compatibility constant defined in \cite{Sara2007}.

\section{The oracle property of the square-root regression problem with a
generalized
elastic-net regularization}
\label{sec:oracle property}
We first consider the following convex problem without $q$ in \eqref{eq:prob}, that is
\begin{eqnarray}\label{primal-problem-convex}
\min_{\beta\in\R^{n}}\Big\{\|X\beta-b\|+ \lambda p(\beta)\Big\}.
\end{eqnarray}
By adding proximal terms, we shall analyze the oracle property of the
square-root regression problem with a generalized elastic-net regularization. For given $\tauone\geq 0$ and $ \tautwo\geq 0$, it takes the following form
\begin{eqnarray}
\min_{\beta\in\R^{n}}\left\{\|X\beta-b\|+ \lambda p(\beta)+
\frac{\tauone}{2}\norm{\beta}^{2}+\frac{\tautwo}{2}\|X\beta-b\|^{2}\right\},\label{eq:elastic
net}
\end{eqnarray}
whose optimal solution set is denoted as $\Omega(\tauone,\tautwo)$.
The purpose of this section is to study the
oracle property of an estimator $\hat{\beta}\in\Omega(\tauone,\tautwo)$
(whose residual is given by $\hat{\varepsilon}:=b-X\hat{\beta}$) to evaluate
how good the estimator is in estimating the true vector $\truebeta$.

For the given norm $p$, the dual norm of $p$ is given by
\begin{eqnarray*}
\pstar(\beta):=\max_{z\in\R^{n}}\Big\{\langle z,\beta\rangle\ |\ p(z)\leq 1\Big\},\ \forall\,\beta\in\R^{n}.
\end{eqnarray*}
For a weakly decomposable norm $p$ with the allowed set $S$, we let
\begin{eqnarray}\label{eq:notation-lambda}
&n_{p}=\frac{\lambda p(\truebeta)}{\norm{\varepsilon}},\
\lambda_{0}=
\frac{\pstar(\varepsilon^{T}X)}{\norm{\varepsilon}},\
\lambda_{m}=\max\left\{\frac{\pstar^{\Sc}((\varepsilon^{T}X)^{\Sc})}{\|\varepsilon\|},\frac{\pstar((\varepsilon^{T}X)_{S})}{\|\varepsilon\|},\pstar^{\Sc}(\truebetaSc),\pstar(\truebeta_{S})\right\},&
\\[5pt]
&
t_{1} = 1+\frac{\tautwo}{2}\|\varepsilon\|+\frac{\tauone p_{*}(\truebeta)p(\truebeta)}{2\|\varepsilon\|}, \quad
t_{2} = 2+\tau+\frac{\sigma p_{*}(\truebeta)p(\truebeta)}{\|\varepsilon\|},
\\[5pt]
&
c_u = t_{1}+n_{p}, \quad
a = \Big( \lam_0 +\tauone \pstar(\truebeta)c_u\Big)\frac{t_{1}}{\lam},
&
\label{eq:notation-a}
\end{eqnarray}
where $\pstar^{\Sc}$ denotes the dual norm of $p^{\Sc}$.

Next we state two basic assumptions which are similar to
those in \cite{Sara2017}. The first assumption is about non-overfitting in the sense that if the optimal solution
$\hat{\beta}$ of \eqref{sec:oracle property}
satisfies $\|\hat{\varepsilon}\|=0$, then it overfits. Furthermore, it has been proved in \cite{Pillo1988} that there exists a scalar $\lambda^{*}$ such that the solution of the problem \eqref{eq:elastic
net} satisfies $X\hat{\beta}-b\neq 0$ if $\lambda>\lambda^{*}$. In other words, one can find the parameter $\lambda$ to avoid overfitting.
\begin{assumption}
\label{assump-1} We assume that $\norm{\hat{\varepsilon}}\neq 0$ and $a+\frac{2\lambda_{0}n_{p}}{\lambda}<1$, where  the constant $a$ is defined in (\ref{eq:notation-a}).
\end{assumption}

\begin{assumption}
\label{assump-2}
The function $p$ is a norm function and weakly decomposable in $\R^{n}$ for
an index set $S\subset\{1,\ldots,n\}$, i.e., there exists a norm
$p^{\Sc}$ defined on $\R^{|\Sc|}$ such that
\begin{eqnarray*}
p(\beta) \geq p(\beta_{S}) +
p^{\Sc}(\betaSc),\quad \forall\, \beta\in\R^{n}.
\end{eqnarray*}
\end{assumption}

\begin{remark}
If $\sigma=0$ and $\tau=0$, then Assumption \ref{assump-1} goes back to Assumption I given by \cite{Sara2017}. Let $s>0$ and $P\left(\|\varepsilon\|\leq s\right)$ be the probability of $\|\varepsilon\|\leq s$. It follows from the comment to Lemma 1 of \cite{LaurentMassart2000} that $P\left(\|\varepsilon\|\leq \varsigma\sqrt{n-2\sqrt{ns}}\right)\leq e^{-s}$ and $P\left(\|\varepsilon\|\geq \varsigma\sqrt{n+2\sqrt{ns}+2s}\right)\leq e^{-s}$. Noticing that for $\alpha_{1}>e^{-n/4}$, $\alpha_{2}>0$ and $\alpha_{1}+\alpha_{2}<1$, $P\left(\bar{\chi}_{1}\leq\|\varepsilon\|\leq \bar{\chi}_{2}\right)\geq 1-\alpha_{1}-\alpha_{2}$
with $\bar{\chi}_{1}:=\varsigma\sqrt{n-2\sqrt{-n\ln(\alpha_{1})}}$ and $\bar{\chi}_{2}:=\varsigma\sqrt{n+2\sqrt{-n\ln(\alpha_{2})}-2\ln(\alpha_{2})}$,
the inequalities $n_{p}\leq \lambda p(\truebeta)/\bar{\chi}_{1}$, $t_{1}\leq 1+\tau\bar{\chi}_{2}/2+\sigma p_{*}(\truebeta)p(\truebeta)/(2\bar{\chi}_{1}):=\bar{t}_{1}, c_{u}\leq \bar{t}_{1}+\lambda p(\truebeta)/\bar{\chi}_{1}$, $a\leq\left(\bar{\lambda}_{0}+\sigma p_{*}(\truebeta)\bar{t}_{1}+\lambda\sigma p_{*}(\truebeta)p(\truebeta)/\bar{\chi}_{1}\right)\bar{t}_{1}/\lambda$ hold with probability at least $1-\alpha_{1}-\alpha_{2}$. It has been proven in section 3.4 of \cite{Sara2017} that $\max\left\{\pstar^{\Sc}((\varepsilon^{T}X)^{\Sc})/\|\varepsilon\|,\pstar((\varepsilon^{T}X)_{S})/\|\varepsilon\|\right\}\leq\bar{\lambda}_{0}$ holds with high probability for some $\bar{\lambda}_{0}$ (may depend on $\sqrt{m}$ and $\sqrt{ln(n)}$). Combining with the result $\lambda_{0}\leq\bar{\lambda}_{0}$ by Lemma 3.1 (see below), $\bar{\chi}_{1}>2\bar{\lambda}_{0}p(\truebeta)$ is valid for high dimensional problems with high probability. We can choose $\tau$, $\sigma$ and $\lambda$ with
\begin{eqnarray*}
\bar{\chi}_{1}&>&\sigma p_{*}(\truebeta)p(\truebeta)\bar{t}_{1}+2\bar{\lambda}_{0}p(\truebeta),\\ \lambda&>&\max\left\{\frac{\bar{\chi}_{1}(\bar{\lambda}_{0}\bar{t}_{1}+\sigma p_{*}(\truebeta)\bar{t}_{1}^{2})}{\bar{\chi}_{1}-\sigma p_{*}(\truebeta)p(\truebeta)\bar{t}_{1}-2\bar{\lambda}_{0}p(\truebeta)},\lambda^{*}\right\}
\end{eqnarray*}
such that Assumption \ref{assump-1} holds with high probability with respect to the
random noise vector $\varepsilon$.
Assumption \ref{assump-2} was also used in \cite{Bach2012} and \cite{Sara2014}.
\end{remark}

We also need some lemmas in order to prove the main theorem. First we introduce the following basic relationship between $\lambda_{0}$ and $\lambda_{m}$.\footnote{Actually in \cite{Sara2017} Stucky and van de Geer once employed this relationship without a proof. For the sake of clarity, we present this fact in the form of a lemma here and its proof in the appendix.}
\begin{lemma}
Let $S$ be an allowed set of a weakly decomposable norm $p$. For the parameters $\lambda_{0}$ and $\lambda_{m}$ defined by (\ref{eq:notation-lambda}), we have $\lambda_{0}\leq\lambda_{m}$ and $p_{*}(\truebeta)\leq\lambda_{m}$.
\end{lemma}
The following lemma, which is from \cite{Sara2014}, shows that
$p(\beta_{S})$ is bounded by $\|X\beta\|$. 
\begin{lemma}\label{lemma:sara}
Given $X\in\mathbb{R}^{m\times n}$. Let $S$ be an allowed set of a weakly decomposable norm $p$ and
$L>0$ be a constant. Then the $p$-eigenvalue can be expressed in the
 following form:
\begin{eqnarray*}
\delta_{p}(L,S)=\min\left\{\frac{\|X\beta\|}{p(\beta_{S})} \; \Big| \;
\beta\in\mathbb{R}^{n}\ \mbox{satisfies the $(L,S)$-cone condition and}\ \beta_{S}\neq
0\right\}.
\end{eqnarray*}
That is $p(\beta_{S})\leq \Gamma_{p}(L,S)\|X\beta\|$.
\end{lemma}
An upper bound of
$\hat{\varepsilon}^{T}X(\truebeta-\hat{\beta})$, a lower bound and an
upper bound of $\|\hat{\varepsilon}\|$ are also important. They are presented in
the following two lemmas.
\begin{lemma}
Suppose that Assumption \ref{assump-1} holds. For the estimator $\hat{\beta}$ of the
generalized elastic-net square-root regression problem (\ref{eq:elastic
net}), we have
\begin{eqnarray*}
\hat{\varepsilon}^{T}X(\truebeta-\hat{\beta})\leq\left(\tautwo+\frac{1}{\|\hat{\varepsilon}\|}\right)^{-1}\left(\lambda
p(\truebeta)+\tauone\pstar(\truebeta)p(\hat{\beta})\right).
\end{eqnarray*}
\end{lemma}

\begin{lemma}
Suppose that Assumption \ref{assump-1} holds. We have
\begin{eqnarray*}
c_{l}:=\frac{1-a-\frac{2\lambda_{0}n_{p}}{\lambda}}{2+\left(1+\frac{\tauone\pstar(\truebeta)}{\lambda}\right)n_{p}}
\; \leq\;
\frac{\|\hat{\varepsilon}\|}{\|\varepsilon\|}
\;\leq\;
c_{u},
\end{eqnarray*}
where the constants $c_u$ and $a$ are defined in \eqref{eq:notation-a}.
\end{lemma}
Based on the above lemmas, we can present the following sharp oracle
inequality on the prediction error.
\begin{theorem}\label{thm:oracle-property}
Let $\delta\in[0,1)$. Under Assumptions \ref{assump-1} and \ref{assump-2}, assume that $\frac{s_{2}-\sqrt{s_{2}^{2}-4s_{1}s_{3}}}{2s_{1}}<\lambda<\frac{s_{2}+\sqrt{s_{2}^{2}-4s_{1}s_{3}}}{2s_{1}}
$
with $s_{1}=\frac{\sigma\lambda_{m}p^{2}(\truebeta)}{\|\varepsilon\|^{2}}$, $s_{2}=1-\frac{\lambda_{m}(3+2\sigma t_{1}+\sigma t_{2})p(\truebeta)}{\|\varepsilon\|}>0$ and $s_{3}=\lambda_{m}(t_{1}+t_{2}+\sigma t_{1}t_{2}+\sigma t_{1}^{2})$.
For any $\hat{\beta}\in\Omega(\tauone,\tautwo)$,
and any $\beta\in \R^n$ such that $\mbox{supp}(\beta)$ is a
subset of $S$, we have that
\begin{eqnarray*}
&&\|X(\hat{\beta}-\truebeta)\|^{2}+2\delta\left((\hat{\lambda}-\tilde{\lambda}_{m})
p^{\Sc}(\hat{\beta}^{\Sc})+(\tilde{\lambda}
+\tilde{\lambda}_{m})p(\hat{\beta}_{S}-\beta)\right)\|\varepsilon\|\\
&&\leq\|X(\beta-\truebeta)\|^{2}+ \Big((1+\delta)(\tilde{\lambda}+\tilde{\lambda}_{m})\Gamma_{p}(L_{S},S)\|\varepsilon\| \Big)^{2}+2\tauone c_{u}\|\hat{\beta}-\truebeta\|\|\beta-\truebeta\|\|\varepsilon\|,
\end{eqnarray*}
where
\begin{eqnarray*}
\hat{\lambda}:=\frac{\lambda c_{l}}{1+\tautwo c_{l}},\ \tilde{\lambda}_{m}:=\lambda_{m}(1+\tauone c_{u}),\ \tilde{\lambda}:=\lambda c_{u},\
L_{S}:=\frac{\tilde{\lambda}+\tilde{\lambda}_{m}}{\hat{\lambda}-\tilde{\lambda}_{m}}\cdot\frac{1+\delta}{1-\delta}.
\end{eqnarray*}
\end{theorem}

An important special case of Theorem \ref{thm:oracle-property} is to choose $\beta=\truebeta$ with ${\rm supp}(\truebeta) \subseteq S$ allowed. Then only the $p$-effective sparsity term $\Gamma_p(L_S,S)$ appears in the upper bound.
\begin{remark}
Since $\lim\limits_{\sigma\downarrow 0}\left(s_{2}^{2}-4s_{1}s_{3}\right)=\left(1-\frac{3\lambda_{m}p(\truebeta)}{\|\varepsilon\|}\right)^{2}>0$, we can find some $\tilde{\sigma}>0$ such that $s_{2}^{2}-4s_{1}s_{3}>0$ holds if $\sigma<\tilde{\sigma}$. Theorem 3.1 is nearly the same as that in \cite{Sara2017} due to $\lim\limits_{\sigma\downarrow 0,\tau\downarrow0}\frac{s_{2}-\sqrt{s_{2}^{2}-4s_{1}s_{3}}}{2s_{1}}=\frac{3\lambda_{m}\|\varepsilon\|}{\|\varepsilon\|-3\lambda_{m}p(\truebeta)}$ and $\lim\limits_{\sigma\downarrow 0,\tau\downarrow0}\frac{s_{2}+\sqrt{s_{2}^{2}-4s_{1}s_{3}}}{2s_{1}}=+\infty$ with a different definition of $\lambda_{m}$.
\end{remark}
\begin{remark}
From Theorem 3.1 we can see that the upper bound is related to some random parts $\lambda_{m}$ and $\|\varepsilon\|$. If we have Gaussian errors $\varepsilon\sim\mathcal{N}(0,\varsigma^{2}I)$, then we know from  Proposition 11 of \cite{Sara2017} that there exists an upper bound $\lambda^{u}_{m}$ of $\lambda_{m}$ such that $\lambda_{m}\leq\lambda_{m}^{u}$ is valid with probability $1-\alpha$ for a given constant $\alpha$. Furthermore, it follows from \cite{LaurentMassart2000} that we can find an upper bound $c_{1}\varsigma$ and a lower bound $c_{2}\varsigma$ for $\|\epsilon\|$ with a high probability. That is, if $\lambda_{m}$ is replaced by $\lambda_{m}^{u}$ and $\|\epsilon\|$ is replaced by $c_{1}\varsigma$ or $c_{2}\varsigma$, then the sharp oracle bound with the Gaussian errors holds with a high probability.
\end{remark}

\section{The proximal majorization-minimization algorithm}
\label{sec:two-stage algorithm}
To deal with the  nonconvexity of the  regularization function in the square-root regression problem (\ref{eq:prob}), we design a two stage proximal majorization-minimization (PMM) algorithm and solve a series of convex subproblems. In stage I, we first solve a problem
 by removing $\ptwo$ from the original problem and adding
appropriate proximal terms, to obtain an initial point
to warm-start the algorithm in the second stage.
In stage I\!I, a series of majorized subproblems are
solved to obtain a solution point.

The basic idea of the PMM algorithm is to linearize the concave term $-q(\beta)$ in the objective
function of \eqref{eq:prob} at each iteration with respect to the current iterate, say
$\tilde{\beta}$. By doing so, the subproblem in each iteration is a convex minimization
problem, which must be solved efficiently in order for the overall algorithm to
be efficient. However, the objective function of the resulting subproblem contains the sum of
two nonsmooth terms ($\norm{X\beta-b}$ and $p(\beta)$), and it is not obvious how such a problem can be solved efficiently.
One important step we take in this paper is to add the proximal
term $\frac{\tautwo}{2}\norm{X\beta- X\tilde{\beta}}^2$ to the objective function
of the subproblem. Through this novel proximal term, the dual of the majorized
 subproblem can then be written explicitly  as an unconstrained optimization problem.
Moreover, its structure is highly conducive for one to
 apply the semismooth Newton (SSN) method to compute an approximate solution
 via  solving a nonlinear system of equations.

\subsection{A semismooth Newton method for the subproblems}

For the purpose of our algorithm developments, given $\tauone>0$, $\tautwo>0$, $\tilde{\beta}\in\R^{n}$, $\tilde{v}\in\R^{n}$, and $\tilde{b}\in\R^{m}$, we consider the following minimization problem:
\begin{eqnarray}\label{primal-subproblem}
\min_{\beta\in\R^{n}}\Big\{h(\beta;\tauone,\tautwo,\tilde{\beta},\tilde{v},\tilde{b})&:=&\|X\beta-b\|+\lambda \pone(\beta)-\ptwo(\tilde{\beta})-\langle\tilde{v},\beta-\tilde{\beta}\rangle\nonumber\\
&&+\frac{\tauone}{2}\|\beta-\tilde{\beta}\|^{2}
+\frac{\tautwo}{2}\|X\beta-\tilde{b}\|^{2}\Big\}.
\end{eqnarray}
The optimization problem (\ref{primal-subproblem}) is equivalent to
\begin{eqnarray}\label{primal-subproblem-1}
\min_{\beta\in\R^{n},y\in\R^{m}}\left\{\left.\|y\|+\lambda \pone(\beta)-\langle\tilde{v},\beta-\tilde{\beta}\rangle+\frac{\tauone}{2}\|\beta-\tilde{\beta}\|^{2}+\frac{\tautwo}{2}\|y+b-\tilde{b}\|^{2}\ \right|\ X\beta-y=b\right\}.
\end{eqnarray}
The dual of  problem (\ref{primal-subproblem-1}) admits the following equivalent minimization form:
\begin{eqnarray}\label{dual-subproblem}
&&\min_{u\in\R^{m}}\Big\{\varphi(u):=  \langle u,b\rangle
+ \frac{\tautwo}{2}\norm{\tautwo^{-1}u+\tilde{b}-b}^{2}
-\|\mbox{Prox}_{\tautwo^{-1}\|\cdot\|}(\tautwo^{-1}u+\tilde{b}-b)\|\\
&&\quad\quad\
-\frac{1}{2\tautwo}\|\mbox{Prox}_{\tautwo\delta_{B}}(u +\tautwo(\tilde{b}-b))\|^{2}
+\frac{\tauone}{2}\|\tilde{\beta}+\tauone^{-1}(\tilde{v}-X^{T}u)\|^{2}\nonumber\\
&&\quad\quad\ -\lambda\, \pone\left(\mbox{Prox}_{\tauone^{-1}\lambda \pone}\big(
\tilde{\beta}+\tauone^{-1}(\tilde{v}-X^{T}u)\big)\right)-\frac{1}{2\tauone}\|\mbox{Prox}_{\tauone(\lambda \pone)^{*}}(\tauone\tilde{\beta}+\tilde{v}-X^{T}u)\|^{2} \Big\},\nonumber
\end{eqnarray}
where $B=\left\{x\ |\ \|x\|\leq 1\right\}$.
Let $\bar{u}:=\mathop{\mbox{argmin}}\limits_{u\in\R^{m}}\varphi(u)$. Then  the optimal solutions $\bar{y},\bar{\beta}$ to the primal problem (\ref{primal-subproblem-1}) can be computed by
\begin{eqnarray*}
\bar{y}=\mbox{Prox}_{\tautwo^{-1}\|\cdot\|}(\tautwo^{-1}\bar{u}+\tilde{b}-b),\quad \bar{\beta}=\mbox{Prox}_{\tauone^{-1}\lambda \pone}\big(\tilde{\beta}
+\tauone^{-1}(\tilde{v}-X^{T}\bar{u})\big).
\end{eqnarray*}
Here we should emphasize  the novelty
in adding the proximal term $\frac{\tautwo}{2}\|X\beta-\tilde{b}\|^{2}$ in
\eqref{primal-subproblem}. Without this term, the objective function in the dual problem
\eqref{dual-subproblem} does not admit an analytical expression. As the reader may
observe later in the next paragraph, without the analytical expression given in \eqref{dual-subproblem}, the
algorithmic development in the rest of this subsection would break down. As a result,
designing the PMM algorithm in the next subsection for
solving \eqref{eq:prob} would also become impossible.

By Moreau's identity (\ref{Moreau-identity}) and the differentiability of the Moreau envelope functions of $\|\cdot\|$ and $\lambda p$, we know that the function
$\varphi$ is convex and continuously differentiable and
\begin{eqnarray*}
\nabla\varphi(u)=\mbox{Prox}_{\tautwo^{-1}\|\cdot\|}(\tautwo^{-1}u+\tilde{b}-b)-X\mbox{Prox}_{\tauone^{-1}\lambda \pone}\big(\tilde{\beta}
+\tauone^{-1}(\tilde{v}-X^{T}u)\big)+b.
\end{eqnarray*}
Thus $\bar{u}$ can be obtained via solving the following nonlinear system of equations:
\begin{eqnarray}\label{dual-subproblem-nonlinear-equation}
\nabla\varphi(u)=0.
\end{eqnarray}
In the rest of this subsection, we will discuss how we can
 apply the SSN method to compute an approximate solution of (\ref{dual-subproblem-nonlinear-equation}) efficiently. Since the mappings $\mbox{Prox}_{\tauone^{-1}\|\cdot\|}(\cdot)$ and
$\mbox{Prox}_{\tautwo^{-1}\lambda \pone}(\cdot)$ are Lipschitz
continuous, the following multifunction is well defined:
\begin{eqnarray*}
\hat{\partial}^{2}\varphi(u):=\tauone^{-1}X\partial\mbox{Prox}_{\tauone^{-1}\lambda
\pone}(\tilde{\beta}+\tauone^{-1}(\tilde{v}-X^{T}u))X^{T}+\tautwo^{-1}\partial\mbox{Prox}_{\tautwo^{-1}\|\cdot\|}(\tautwo^{-1}u+\tilde{b}-b).
\end{eqnarray*}
 Let
$U\in\partial\mbox{Prox}_{\tauone^{-1}\lambda
\pone}(\tilde{\beta}+\tauone^{-1}(\tilde{v}-X^{T}u))$ and
$V\in\partial\mbox{Prox}_{\tautwo^{-1}\|\cdot\|}(\tautwo^{-1}u+\tilde{b}-b)$. Then we have
$H:=\tauone^{-1}XUX^{T}+\tautwo^{-1}V\in\hat{\partial}^{2}\varphi(u)$.
The following proposition demonstrates that $H$
is nonsingular at the solution point that does not over fit, which, under  Assumption \ref{assump-1},  holds automatically  when it is close to $\hat \beta$.
This property is crucial to ensure the fast convergence of the
SSN method  for computing an approximate solution of \eqref{dual-subproblem}.

\begin{proposition}\label{prop;4.1}
Suppose that the unique optimal solution $\bar \beta $ to the problem (\ref{primal-subproblem}) satisfies $\|X\bar{\beta}-b\|\neq 0$. Then
all the elements of
$\hat{\partial}^{2}\varphi(\bar{u})$ are  positive
definite.
\end{proposition}
\begin{proof}
By the assumption,
$\bar{y}=X\bar{\beta}-b\neq0$. Furthermore,
$\bar{y}=\mbox{Prox}_{\tautwo^{-1}\|\cdot\|}(\tilde{u})=\tilde{u}-\tautwo^{-1}\Pi_{B}(\tautwo\tilde{u})$,
where $\tilde{u}=\tautwo^{-1}\bar{u}+\tilde{b}-b$, $\Pi_{B}$ is the
Euclidean projection operator onto $B$.
Since $\bar{y}\not=0$, it follows that $\|\tilde{u}\|>\frac{1}{\tautwo}$ and
$\mbox{Prox}_{\tautwo^{-1}\|\cdot\|}(\tilde{u})$ is
differentiable with
\begin{eqnarray*}
V:=\nabla\mbox{Prox}_{\tautwo^{-1}\|\cdot\|}(\tilde{u})
=\left(1-\frac{1}{\tautwo\|\tilde{u}\|}\right)\mathcal{I}_{m}+\frac{\tilde{u}\tilde{u}^{T}}{\tautwo\|\tilde{u}\|^{3}}.
\end{eqnarray*}
Hence for any $U\in\partial\mbox{Prox}_{\tauone^{-1}\lambda
\pone}(\tilde{\beta}+\tauone^{-1}(\tilde{v}-X^{T}\bar{u}))$,
$H=\tauone^{-1}XUX^{T}+\tautwo^{-1}V\in\hat\partial^{2}\varphi(\bar{u})$.
Since $V$ is positive definite and $XUX^{T}$ is positive
semidefinite, $H$ is positive definite. This completes the proof.
\end{proof}
Now we discuss how to apply the SSN
method to solve the nonsmooth equation
(\ref{dual-subproblem-nonlinear-equation}) to obtain an approximate solution
efficiently. We first prove
that $\nabla\varphi$ is strongly semismooth.
\begin{proposition}\label{prop4.2}
The function $\nabla\varphi$ is strongly semismooth.
\end{proposition}
\begin{proof}
Firstly, we have assumed that the proximal operator
$\mbox{Prox}_{\pone}(\cdot)$ is strongly semismooth. Secondly, by  Proposition 4.3
of \cite{Chen2003}, it is known that the projection operator onto the second order cone is
strongly semismooth. The strongly semismoothness of
the proximal operator $\mbox{Prox}_{\|\cdot\|}(\cdot)$
then follows from Theorem 4
of \cite{Meng2005}, which states that if the projection onto the second order cone
is strongly semismooth, then so is the proximal mapping of $\|\cdot\|$. From here, it is easy to prove the required result and we omit the details.
\end{proof}

Now we can apply the SSN method to solve
(\ref{dual-subproblem-nonlinear-equation}) as follows.
\bigskip

\ni\fbox{\parbox{\textwidth}{\noindent{\bf Algorithm
SSN($\tauone,\tautwo,\tilde{\beta},\tilde{v},\tilde{b}$)
with input  $\tauone>0, \tautwo>0,\tilde{\beta},\tilde{v}\in\R^{n},\tilde{b}\in\R^{m}$.} \label{alg:Newton-CG1}
 Choose
$\mu\in(0,\frac{1}{2}), \ \overline{\eta}\in(0,1),\ \varrho\in(0,1],\ \delta\in(0,1)$,
and $u^{0}\in\R^{m}$. For $j=0,1,\ldots,$ iterate the following steps:
\begin{description}
\item [Step 1.] Choose $U^{j}\in\partial\mbox{Prox}_{\tauone^{-1}\lambda \pone}
(\tilde{\beta}+\tauone^{-1}(\tilde{v}-X^{T}u^{j}))$
and
$V^{j}\in\partial\mbox{Prox}_{\tautwo^{-1}\|\cdot\|}(\tautwo^{-1}u^{j}+\tilde{b}-b)$.
Let $H^{j}=\tauone^{-1}XU^{j}X^{T}+\tautwo^{-1}V^{j}$.
Compute an approximate solution $\Delta u^{j}$ to the linear system
\begin{eqnarray*}
H^{j}\Delta
u=-\nabla\varphi(u^{j})
\end{eqnarray*}
such that
\begin{eqnarray}
\|H^{j}\Delta u^{j}+\nabla\varphi(u^{j})\|\leq
\min\{\overline{\eta},\|\nabla\varphi(u^{j})\|^{1+\varrho}\}.\label{ineq:stopping
criterion}
\end{eqnarray}
\item [Step 2.]  Set $\alpha_{j}=\delta^{t_{j}}$, where $t_{j}$ is the first nonnegative integer $t$ such that
\begin{eqnarray*}
\varphi(u^{j}+\delta^{t}\Delta u^{j})\leq
\varphi(u^{j})+\mu\delta^{t}\langle\nabla
 \varphi(u^{j}),\Delta u^{j}\rangle.&
\end{eqnarray*}
\item [Step 3.]  Set $u^{j+1}=u^{j}+\alpha_{j}\Delta u^{j}$.
\end{description}}}

\bigskip

With Propositions \ref{prop;4.1} and  \ref{prop4.2}, the SSN method can be proven to be globally convergent and locally
superlinearly convergent. One may see Theorem 3.6 of \cite{LiSunToh2018} for the
details. The local convergence rate for Algorithm SSN is stated in the
next theorem without proof.
\begin{theorem}
Suppose that $\|X\bar{\beta}-b\|\neq 0$ holds.
Then the sequence $\{u^{j}\}$   generated by Algorithm SSN
  converges to the unique optimal solution
$\overline{u}$ of the problem (\ref{dual-subproblem}) and
\begin{eqnarray*}
\|u^{j+1}-\overline{u}\| &=&
\mathcal{O}(\|u^{j}-\overline{u}\|)^{1+\varrho}.
\end{eqnarray*}
\end{theorem}

\subsection{The SSN based proximal majorization-minimization algorithm}

In this subsection, we describe the details of the PMM algorithm
for solving \eqref{eq:prob} wherein each subproblem is solved by the SSN method. We briefly present the PMM algorithm as follows.

\bigskip
\ni\fbox{\parbox{\textwidth}{\noindent{\bf Algorithm PMM.} Let $\tauone^{2,0}>0$,
$\tautwo^{2,0}>0$ be given parameters.
\begin{description}
 {
\item [Step 1.] Find $\tauone^{1}>0$,
$\tautwo^{1}>0$ and compute
\begin{eqnarray}\label{stage-I-problem}
\beta^{0}\approx\mathop{\mbox{argmin}}_{\beta\in\R^{n}} \left\{h(\beta;\tauone^{1},\tautwo^{1},0,0,b)\right\}
\end{eqnarray}
via solving its dual problem such that the KKT residual of the  problem (\ref{primal-problem-convex}) satisfies a prescribed stopping criterion. That is, given $(\tauone,\tautwo,\tilde{\beta},\tilde{v},\tilde{b})=(\tauone^{1},\tautwo^{1},0,0,b)$, apply the SSN method to find an approximate solution $u^{0}$ of (\ref{dual-subproblem-nonlinear-equation}) and then set $\beta^{0}=\mbox{Prox}_{\lambda \pone/\tauone^{1}}(-X^{T}u^{0}/\tauone^{1})$.
Let $k=0$ and go to Step 2.1.}
\item [Step 2.1] Compute
\begin{eqnarray*}
\beta^{k+1} \; = \; \mathop{\mbox{argmin}}_{\beta\in\R^{n}} \left\{h(\beta;\tauone^{2,k},\tautwo^{2,k},\beta^{k},\nabla \ptwo(\beta^{k}),X\beta^{k})+\langle\delta^{k},\beta-\beta^{k}\rangle\right\}
\end{eqnarray*}
via solving its dual problem. That is, given $(\tauone,\tautwo,\tilde{\beta},\tilde{v},\tilde{b})=(\tauone^{2,k},\tautwo^{2,k},\beta^{k},\nabla \ptwo(\beta^{k}),X\beta^{k})$, apply the SSN method to find an approximate solution $u^{k+1}$ of (\ref{dual-subproblem-nonlinear-equation}) such that the error vector $\delta^k$
satisfies the following accuracy condition:
\begin{align}\label{convergence-assump}
\|\delta^{k}\|\leq \frac{\tauone^{2,k}}{4}\|\beta^{k+1}-\beta^{k}\|+\frac{\tau^{2,k}\|X\beta^{k+1}-X{\beta}^{k}\|^{2}}{2\|\beta^{k+1}-\beta^{k}\|},
\end{align}
where $\beta^{k+1}=\mbox{Prox}_{\lambda \pone/\tauone^{2,k}}(\beta^{k}
+(\nabla \ptwo(\beta^{k})-X^{T}u^{k+1})/\tauone^{2,k})$.
\item [Step 2.2.]
If $\beta^{k+1}$ satisfies a prescribed stopping
criterion, terminate; otherwise update $\tauone^{2,k+1}=\rho_{k}\tauone^{2,k}$, $\tautwo^{2,k+1}=\rho_{k}\tautwo^{2,k}$ with $\rho_{k}\in(0,1)$ and return to Step 2.1 with $k=k+1$.
\end{description}}}
\bigskip

Since $h(\beta;\tauone^{1},\tautwo^{1},0,0,b)$ is bounded below, it has been proven in Proposition 4.19 of \cite{Flemming2011} that the optimal objective value of the problem (\ref{stage-I-problem}) will converge to the optimal objective value of the problem (\ref{primal-problem-convex}) with a difference of $\ptwo(0)$ 
as $\tauone^{1}\rightarrow 0$, $\tautwo^{1}\rightarrow 0$. We simply describe the convergence result of the algorithm in our first stage as follows and give a similar proof to that of Proposition 4.19 in \cite{Flemming2011}.
\begin{theorem}
Let $\bar{h}(\tauone^{1},\tautwo^{1}):=\min\limits_{\beta\in\R^{n}}\left\{h(\beta;\tauone^{1},\tautwo^{1},0,0,b)\right\}$. Then we have
\begin{eqnarray*}
\lim_{\tauone^{1},\tautwo^{1}\rightarrow 0}\bar{h}(\tauone^{1},\tautwo^{1})=\min_{\beta\in\R^{n}}\Big\{\|X\beta-b\|+\lambda \pone(\beta)-\ptwo(0)\Big\}.
\end{eqnarray*}
\end{theorem}
\begin{proof}
For any $\tauone^{1},\tautwo^{1}>0$ and $\beta\in\R^{n}$, we have that
\begin{eqnarray*}
\bar{h}(\tauone^{1},\tautwo^{1})\leq \|X\beta-b\|+\lambda \pone(\beta)-\ptwo(0)+\frac{\tauone^{1}}{2}\|\beta\|^{2}+\frac{\tautwo^{1}}{2}\|X\beta-b\|^{2}.
\end{eqnarray*}
Therefore, $\lim\limits_{\tauone^{1},\tautwo^{1}\rightarrow 0}\bar{h}(\tauone^{1},\tautwo^{1})\leq\|X\beta-b\|+\lambda \pone(\beta)-\ptwo(0)$. That is \begin{eqnarray*}
\lim\limits_{\tauone^{1},\tautwo^{1}\rightarrow 0}\bar{h}(\tauone^{1},\tautwo^{1})\leq\min_{\beta\in\R^{n}}\Big\{\|X\beta-b\|+\lambda \pone(\beta)-\ptwo(0)\Big\}.
\end{eqnarray*}
Furthermore, $\bar{h}(\tauone^{1},\tautwo^{1})\geq\min\limits_{\beta\in\R^{n}}\Big\{\|X\beta-b\|+\lambda \pone(\beta)-\ptwo(0)\Big\}$. The desired result follows.
\end{proof}

\subsection{Convergence analysis of the PMM algorithm}
\label{sec:convergence analysis}

In this subsection, we analyze the convergence of the PMM algorithm.
First we recall the definition of the KL property of a function (see
\cite{attouch-bolte2009,bolte-pauwels2016,bolte-sabach-teboulle2014}
for more details). Let $\eta>0$ and $\Phi_{\eta}$ be the set of all
concave functions $\psi:[0,\eta)\rightarrow\R_{+}$ such
that
\begin{description}
\item (1) $\psi(0)=0$;
\item (2) $\psi$ is continuous at $0$ and continuously differentiable
on $(0,\eta)$;
\item (3) $\psi'(x)>0$, for any $x\in(0,\eta)$.
\end{description}
\begin{definition}
Let $f:\R^{n}\rightarrow (-\infty, \infty]$ be a proper lower
semi-continuous function and $\bar{x}\in\mbox{dom}(\partial
f):=\{x\in\mbox{dom}(f)\, |\partial f(x)\neq \emptyset\}$. The function $f$
is said to have the KL property at $\bar{x}$ if there exist
$\eta>0$, a neighbourhood $\mathcal{U}$ of $\bar{x}$ and a concave
function $\psi\in\Phi_{\eta}$ such that
\begin{eqnarray*}
\psi'(f(x)-f(\bar{x}))\mbox{dist}(0,\partial f (x))\geq 1,\quad \forall\,
x\in\mathcal{U}\ \mbox{and}\ f(\bar{x})<f(x)<f(\bar{x})+\eta,
\end{eqnarray*}
where $\mbox{dist}(x,C):=\min_{y\in C}\|y-x\|$ is the distance from
a point $x$ to a nonempty closed set $C$.
Furthermore, a function $f$ is called a KL function if it satisfies the KL property at all points in $\mbox{dom}\partial
f$.
\end{definition}
Note that a function is said to have the KL property at
$\bar{x}$ with an exponent $\alpha$ if the function $\psi$ in the
definition of the KL property takes the form of $\psi(x)=\gamma
x^{1-\alpha}$ with $\gamma>0$ and $\alpha\in[0,1)$. For the function $f(x)=x$, it
 has the KL property at any point
with the exponent 0.

Now we are ready to conduct the convergence analysis of the PMM algorithm. Denote $h_{k}(\beta):=h(\beta;\tauone^{2,k},\tautwo^{2,k},\beta^{k},\nabla \ptwo(\beta^{k}),X\beta^{k})$.
At the $k$-th iteration of stage I\!I, we have that
\begin{eqnarray*}
\beta^{k+1} = \mathop{\mbox{argmin}}_{\beta\in\R^{n}}
\left\{h_{k}(\beta)+\langle\delta^{k},\beta-\beta^{k}\rangle\right\}
\end{eqnarray*}
such that  condition (\ref{convergence-assump}) is satisfied.
The following lemma shows the descent property of the function
$h_{k}.$
\begin{lemma}\label{lemma:converge-1}
Let $\beta^{k+1}$ be an approximate solution of the subproblem in the $k$-th iteration such that (\ref{convergence-assump}) holds.
Then we have
\begin{eqnarray*}
h_{k}(\beta^{k})&\geq&h_{k}(\beta^{k+1})-\frac{\sigma^{2,k}}{4}\|\beta^{k+1}-\beta^{k}\|^{2}-\frac{\tau^{2,k}}{2}\|X\beta^{k+1}-X\beta^{k}\|^{2}.
\end{eqnarray*}
\end{lemma}
\begin{proof}
Since $h_{k}$ is a convex function and $-\delta^{k}\in\partial
h_{k}(\beta^{k+1})$, we obtain
\begin{eqnarray*}
h_{k}(\beta^{k})-h_{k}(\beta^{k+1})\geq\langle\delta^{k},\beta^{k+1}-\beta^{k}\rangle\geq-\frac{\sigma^{2,k}}{4}\|\beta^{k+1}-\beta^{k}\|^{2}-\frac{\tau^{2,k}}{2}\|X\beta^{k+1}-X\beta^{k}\|^{2}.
\end{eqnarray*}
The last inequality is valid since the condition (\ref{convergence-assump}) holds. The desired result follows.
\end{proof}
Next we recall the following lemma which is similar to that in \cite{Cui2018,Pang2016}. 
\begin{lemma}\label{lemma:converge-2}
The vector $\bar{\beta}\in\R^{n}$ is a $d$-stationary point of
(\ref{eq:prob}) if and only if there exist $\tauone,\tautwo\geq0$ such that
\begin{eqnarray*}
\bar{\beta}\in\mathop{\rm{argmin}}_{\beta\in\R^{n}}\Big\{
 h(\beta;\tauone,\tautwo,\bar{\beta},\nabla \ptwo(\bar{\beta}),X\bar{\beta})\Big\}.
\end{eqnarray*}
\end{lemma}
\begin{proof} Recall the objective function $g$ defined in \eqref{eq:prob}.
Since $g$ is directionally differentiable at $\bar{\beta}$, we can see that $\bar{\beta}$ being a d-stationary point of $g$ is equivalent to $0\in\partial g(\bar\beta)$. It is easy
to show that $\partial g(\bar\beta)=\partial_{\beta}h(\bar\beta;\tauone,\tautwo,\bar{\beta},\nabla \ptwo(\bar{\beta}),X\bar{\beta})$. For given $\tauone$, $\tautwo$ and $\bar{\beta}$, the function $h(\cdot;\tauone,\tautwo,\bar{\beta},\nabla \ptwo(\bar{\beta}),X\bar{\beta})$ is convex. Thus $0\in\partial_{\beta} h(\bar\beta;\tauone,\tautwo,\bar{\beta},\nabla \ptwo(\bar{\beta}),X\bar{\beta})$ is equivalent to $\bar\beta\in\mathop{\mbox{argmin}}_{\beta\in\R^{n}}\Big\{h(\beta;\tauone,\tautwo,\bar{\beta},\nabla \ptwo(\bar{\beta}),X\bar{\beta})\Big\}$. This completes the proof.
\end{proof}

It has been proven in \cite{Cui2018} that the sequence generated by
the PMM algorithm converges to a directional stationary solution if
the exact solutions of the subproblems are obtained. The following
theorem shows that the result is also true if the subproblems are
solved approximately.
\begin{theorem}\label{thm:converge}
Suppose that the function
$g$ in \eqref{eq:prob} is bounded below and Assumption \ref{assump-1} holds.
Assume that $\{\tauone^{2,k}\}$ and $\{\tautwo^{2,k}\}$ are convergent
sequences.
Let $\{\beta^{k}\}$ be the
sequence generated by the PMM algorithm. Then every cluster point of the sequence $\{\beta^{k}\}$, if exists, is a d-stationary
point of (\ref{eq:prob}).
\end{theorem}
\begin{proof}
Combing Lemma \ref{lemma:converge-1} and the convexity of $\ptwo$, we have
\begin{eqnarray*}
 \hspace{-7mm}
g(\beta^{k})&=&h_{k}(\beta^{k}) \;\geq\; h_{k}(\beta^{k+1})-\frac{\sigma^{2,k}}{4}\|\beta^{k+1}-\beta^{k}\|^{2}-\frac{\tau^{2,k}}{2}\|X\beta^{k+1}-X\beta^{k}\|^{2}\nonumber\\
&\geq&
g(\beta^{k+1})+\frac{\tauone^{2,k}}{4}\|\beta^{k+1}-\beta^{k}\|^{2}.
\label{eq:q-non-degerate}
\end{eqnarray*}
Therefore the sequence $\{g(\beta^{k})\}$ is non-increasing. Since
$g(\beta)$ is bounded below, the sequence $\{g(\beta^{k})\}$
converges and so is the sequence $\{\|\beta^{k+1}-\beta^{k}\|\}$
which converges to zero. Next, we prove that the limit of a
convergent subsequence of $\{\beta^{k}\}$ is a $d$-stationary point
of (\ref{eq:prob}). Let $\beta^{\infty}$ be the limit of a
convergent subsequence $\{\beta^{k}\}_{k\in\mathcal{K}}$. We can easily prove that $\{\beta^{k+1}\}_{k\in\mathcal{K}}$ also converges to $\beta^{\infty}$. It follows from the definition of $\beta^{k+1}$ that
\begin{eqnarray*}
h_{k}(\beta)&\geq& h_{k}(\beta^{k+1})+\langle\delta^{k},\beta^{k+1}-\beta\rangle\geq h_{k}(\beta^{k+1})-\|\delta^{k}\|
\|\beta^{k+1}-\beta\|,\quad \forall\, \beta\in\R^{m}.
\end{eqnarray*}
Letting $k(\in\mathcal{K})\rightarrow\infty$, we obtain that
\begin{equation*}
\beta^{\infty}\in\mathop{\mbox{argmin}}_{\beta\in\R^{n}}\Big\{h(\beta;\tauone^{2,\infty},\tautwo^{2,\infty},\beta^{\infty},\nabla \ptwo(\beta^{\infty}),X\beta^{\infty})\Big\},
\end{equation*}
where $\tauone^{2,\infty}=\lim\limits_{k\rightarrow\infty}\tauone^{2,k}\geq 0$ and $\tautwo^{2,\infty}=\lim\limits_{k\rightarrow\infty}\tautwo^{2,k}\geq 0$.
The desired result follows from Lemma \ref{lemma:converge-2}. This completes the proof.
\end{proof}
We can also establish the local convergence rate of the sequence
$\{\beta^{k}\}$ under either an isolation assumption of the accumulation
point or the KL property assumption.
\begin{theorem}
Suppose that the function $g$ is bounded below and Assumption \ref{assump-1} holds. Let $\{\beta^{k}\}$ be the
sequence generated by the PMM algorithm.
Let $\mathcal{B}^{\infty}$ be the set of  cluster points of
the sequence $\{\beta^{k}\}$. If either one of the following two
conditions holds,
\begin{enumerate}[(a)]
\item  $\mathcal{B}^{\infty}$ contains an isolated element;
\item The
sequence $\{\beta^{k}\}$ is bounded; for all $\beta\in\mathcal{B}^{\infty}$, $\nabla \ptwo$ is locally Lipschitz continuous near $\beta$
and the
function $g$ has the KL property at
$\beta$;
\end{enumerate}
then the whole sequence $\{\beta^{k}\}$ converges to
an element of $\mathcal{B}^{\infty}$. Moreover, if  condition (b) is
satisfied such that $\{\beta^{k}\}$ converges to
$\beta^{\infty}\in\mathcal{B}^{\infty}$ and the function $g$ has the KL
property at $\beta^{\infty}$ with an exponent $\alpha\in[0,1)$, then
we have the following results:
\begin{enumerate}[(i)]
\item If $\alpha=0$, then the sequence $\{\beta^{k}\}$ converges in a
finite number of steps;
\item If $\alpha\in(0,\frac{1}{2}]$, then the sequence $\{\beta^{k}\}$
converges R-linearly, that is, for all $k\geq 1$, there exist $\nu>0$
and $\eta\in[0,1)$ such that
$\|\beta^{k}-\beta^{\infty}\|\leq\nu\eta^{k}$;
\item If $\alpha\in(\frac{1}{2},1)$, then the sequence
$\{\beta^{k}\}$ converges R-sublinearly, that is, for all $k\geq 1$,
there exists $\nu>0$ such that
$\|\beta^{k}-\beta^{\infty}\|\leq\nu
k^{-\frac{1-\alpha}{2\alpha-1}}$.
\end{enumerate}
\end{theorem}
\begin{proof}
We know from Theorem \ref{thm:converge} that
$\lim\limits_{k\rightarrow\infty}\|\beta^{k+1}-\beta^{k}\|=0$. Then it
follows from Proposition 8.3.10 of \cite{Facchinei2003} that the
sequence $\{\beta^{k}\}$ converges to an isolated element of
$\mathcal{B}^{\infty}$ under the condition (a). In order to derive
the convergence rate of the sequence $\{\beta^{k}\}$ under the
condition (b), we first establish some properties of the
sequence $\{\beta^{k}\}$, i.e.,
\begin{enumerate}[(1)]
\item $g(\beta^{k})\geq
g(\beta^{k+1})+\frac{\tauone^{2,k}}{4}\|\beta^{k+1}-\beta^{k}\|^{2}$;
\item there exists a subsequence $\{\beta^{k_{j}}\}$ of $\{\beta^{k}\}$
such that $\beta^{k_{j}}\rightarrow \beta^{\infty}$ with
$g(\beta^{k_{j}})\rightarrow g(\beta^{\infty})$ as $j\rightarrow
\infty$;
\item for $k$ sufficient large, there exist a constant $K>0$ and $\xi^{k+1}\in\partial g(\beta^{k+1})$ such
that $\|\xi^{k+1}\|\leq K\|\beta^{k+1}-\beta^{k}\|$.
\end{enumerate}
The properties (1) and (2) are already known from Theorem
\ref{thm:converge}. To establish the property (3), we first note that $\mathcal{B}^{\infty}$ is a nonempty, compact and connected set by Proposition 8.3.9 of \cite{Facchinei2003}. Furthermore, let
$\xi^{k+1}=\nabla \ptwo(\beta^{k})-\nabla
\ptwo(\beta^{k+1})-\tauone^{2,k}(\beta^{k+1}-\beta^{k})-\tautwo^{2,k}X^{T}X(\beta^{k+1}-\beta^{k})-\delta^{k}$. We have that $\xi^{k+1}\in\partial g(\beta^{k+1})$. Since $\nabla \ptwo$ is locally Lipschitz continuous near all
$\beta\in\mathcal{B}^{\infty}$ and $\|\delta^{k}\|\leq
\frac{\sigma^{2,k}}{4}\|\beta^{k+1}-\beta^{k}\|+\frac{\tau^{2,k}\|X\beta^{k+1}-X\beta^{k}\|^{2}}{2\|\beta^{k+1}-\beta^{k}\|}$, the property (3) holds for some
constant $K>0$ with $\|\xi^{k+1}\|\leq K\|\beta^{k+1}-\beta^{k}\|$ for $k$ sufficiently large.
 With the properties (1)-(3), the convergence rate of
the sequence $\{\beta^{k}\}$ can be established similarly to that of
Proposition 4 of \cite{bolte-pauwels2016}.
\end{proof}

\section{Numerical experiments}
\label{sec:Numerical issues} In this section, we use
some
numerical experiments to demonstrate the efficiency of our PMM algorithm for the square-root
regression problems. We
implemented the algorithm in MATLAB R2017a. All runs were performed
on a PC (Intel Core 2 Duo 2.6 GHz with 4 GB RAM).
We tested our algorithm on two types of data sets. The first set consists of synthetic data generated randomly in the high-sample-low-dimension setting. That is,
\begin{eqnarray*}
b=X\truebeta+\varsigma\varepsilon,\quad\varepsilon\sim N(0,I).
\end{eqnarray*}
Each row of the input data $X\in\R^{m\times n}$ is generated randomly from the multivariate normal
distribution $N(0,\Sigma)$ with $\Sigma$ as the covariance matrix.
Now we present four examples which are similar to that in \cite{Zuohui2005}.
For each instance, we generate 8000 observations for the training data set and 2000 observations for the validation data set.
\begin{enumerate}
\item[(a)]In example 1, the problem has 800 predictors. Let $\beta=(3,1.5,0,0,2,0,0,0)$ and $\truebeta=(\underbrace{\beta,\ldots,\beta}_{100})^{T}$. The parameter $\varsigma$ is set to  $3$ and the pairwise correlation between the $i$-th predictor and the $j$-th predictor is set to be $\Sigma_{ij}=0.5^{|i-j|}$.
\item[(b)]In example 2, the setting is the same as that in example 1 except that $\truebeta=(\underbrace{\beta,\ldots,\beta}_{400})^{T}$ with the vector $\beta=(0,1)$.
\item[(c)]In example 3, we set $\truebeta=(\underbrace{\beta,\ldots,\beta}_{200})^{T}$ with the vector $\beta=(0,1)$, $\varsigma=15$ and $\Sigma_{ij}=0.8^{|i-j|}$.
\item[(d)]In example 4, the problem has 800 predictors. We choose $\truebeta=(\underbrace{3,\ldots,3}_{300},\underbrace{0,\ldots,0}_{500})$ and $\varsigma=3$. Let $X_{i}$ be the $i$-th predictor of $X$.
For $i\leq 300$, $X_i$ is  generated as follows:
    \begin{eqnarray*}
    X_{i}&=&Z_{1}+\tilde{\varepsilon}_{i},\quad Z_{1}\sim N(0,I),\quad i=1,\ldots,100,\\
    X_{i}&=&Z_{2}+\tilde{\varepsilon}_{i},\quad Z_{2}\sim N(0,I),\quad i=101,\ldots,200,\\
    X_{i}&=&Z_{3}+\tilde{\varepsilon}_{i},\quad Z_{3}\sim N(0,I),\quad i=201,\ldots,300,
    \end{eqnarray*}
    with $\tilde{\varepsilon}_{i}\sim N(0,0.01I)$, $i=1,\ldots,300$.
    For $i > 300$, the predictor $X_{i}$ is just white noise, i.e., $X_{i}\sim N(0,I)$.
\end{enumerate}

We also evaluate our
algorithm on some large scale LIBSVM data sets $(X,b)$
\cite{Chang2011} which are obtained from the UCI data repository
\cite{Lichman}. As in \cite{LiSunToh2018}, we use the
method in \cite{Huang2010} to expand the features of these data sets
by using polynomial basis functions. The last digit in the names
of the data sets, abalone7, bodyfat7, housing7, mpg7 and space9,
indicate the order of the polynomial used to expand the features.
The number of nonzero elements of a vector is defined as the minimal
$k$ such that
\begin{eqnarray*}
\sum_{i=1}^{k}|\check{\beta}_{i}|\geq 0.9999\|\beta\|_{1},
\end{eqnarray*}
where $\check{\beta}$ is obtained by sorting $\beta$ such that
$|\check{\beta}_{1}|\geq|\check{\beta}_{2}|\geq\ldots\geq|\check{\beta}_{n}|$.

In all the experiments, the parameter $\lambda$ is set to $\lambda=\lambda_{c}\Lambda$,
where $\Lambda=1.1\Phi^{-1}(1-0.05/(2n))$ with $\Phi$ being
the cumulative normal distribution function and $\Lambda$ is the theoretical choice recommended in \cite{BelloniCW2011} to compute a specific
coefficient estimate.

\subsection{Numerical experiments for the convex square-root regression problems}

In this section, we compare the performances of the alternating direction method of multipliers
(ADMM) and our stage I algorithm for solving the convex square-root regression
problem (\ref{primal-problem-convex}). For comparison purpose, we adopt the widely used ADMM algorithms for both the primal and dual problems of (\ref{primal-problem-convex}).
For convenience, we use pADMM to denote the ADMM applied to the primal problem, dADMM to denote the ADMM applied to the dual problem, and PMM to denote
our stage I algorithm for solving the convex square-root regression problem (\ref{primal-problem-convex}).

\subsubsection{The ADMM for the problem (\ref{primal-problem-convex})}

In this subsection, we describe the implementation details of the ADMM for the problem (\ref{primal-problem-convex}).
The convex problem (\ref{primal-problem-convex}) can be written equivalently as
\begin{eqnarray}\label{primal-problem-convex-1}
\min_{\beta,z\in\R^{n},y\in\R^{m}}\left\{\|y\|+ \lambda \pone(z)\,\Big|\,X\beta-y=b,\beta-z=0\right\}.
\end{eqnarray}
The dual problem corresponding
to (\ref{primal-problem-convex-1}) has the following form
\begin{eqnarray}\label{dual-problem-convex}
\min_{u,w\in\R^{m},v\in\R^{n}}\left\{\delta_{B}(w)+(\lambda \pone)^{*}(v)+\langle u,b\rangle\Big|X^{T}u+v=0,-u+w=0\right\}.
\end{eqnarray}
Given $\zeta>0$, the augmented Lagrangian functions corresponding to (\ref{primal-problem-convex-1}) and (\ref{dual-problem-convex}) are given by
\begin{eqnarray*}
\mathcal{L}_{\zeta}(\beta,y,z;u,v)&:=&\|y\|+\lambda \pone(z)+\langle u,X\beta-y-b\rangle+\frac{\zeta}{2}\|X\beta-y-b\|^{2}\\
&&+\langle v,\beta-z\rangle+\frac{\zeta}{2}\|\beta-z\|^{2},\\
\tilde{\mathcal{L}}_{\zeta}(u,v,w;\beta,y)&:=&\delta_{B}(w)+(\lambda \pone)^{*}(v)+\langle u,b\rangle-\langle\beta,X^{T}u+v\rangle+\frac{\zeta}{2}\|X^{T}u+v\|^{2}\\
&&-\langle y,-u+w\rangle+\frac{\zeta}{2}\|-u+w\|^{2},
\end{eqnarray*}
respectively. Based on the above augmented Lagrangian functions,
 the ADMMs (\cite{Eckstein1992,Gabay1976}) for solving (\ref{primal-problem-convex-1}) and (\ref{dual-problem-convex}) are given as follows.

\bigskip
\ni\fbox{\parbox{\textwidth}{\noindent{\bf Algorithm pADMM for the primal problem (\ref{primal-problem-convex-1}).}
Let $\rho\in(0,(1+\sqrt{5})/2)$, $\zeta>0$ be given parameters. Choose $(y^{0},z^{0},u^{0},v^{0})\in\R^{m}\times\R^{n}\times\R^{m}\times\R^{n}$,
set $k=0$ and iterate.
\begin{description}
\item [Step 1.] Compute
\begin{eqnarray*}
\beta^{k+1}&=&\mathop{\mbox{argmin}}_{\beta\in\R^{n}}\left\{\mathcal{L}_{\zeta}(\beta,y^{k},z^{k};u^{k},v^{k})\right\}\\
&=&(I_{n}+X^{T}X)^{-1}(z^{k}-\zeta^{-1}v^{k}+X^{T}(y^{k}+b-\zeta^{-1}u^{k})),\\
(y^{k+1},z^{k+1})&=&\mathop{\mbox{argmin}}_{y\in\R^{m},z\in\R^{n}}\left\{\mathcal{L}_{\zeta}(\beta^{k+1},y,z;u^{k},v^{k})\right\}\\
&=&\left(\mbox{Prox}_{\zeta^{-1}\|\cdot\|}(X\beta^{k+1}-b+\zeta^{-1}u^{k}),\mbox{Prox}_{\zeta^{-1}\lambda \pone}(\beta^{k+1}+\zeta^{-1}v^{k})\right).
\end{eqnarray*}
\item [Step 2.] Update
\begin{equation*}
u^{k+1}=u^{k}+\rho\zeta(X\beta^{k+1}-y^{k+1}-b),\
v^{k+1}=v^{k}+\rho\zeta(\beta^{k+1}-z^{k+1}).
\end{equation*}
If the prescribed stopping
criterion is satisfied, terminate; otherwise
return to Step 1 with $k=k+1$.
\end{description}}}

\bigskip
\ni\fbox{\parbox{\textwidth}{\noindent{\bf Algorithm dADMM for the dual problem (\ref{dual-problem-convex}).}
Let $\rho\in(0,(1+\sqrt{5})/2)$, $\zeta>0$ be given parameters. Choose $(v^{0},w^{0},\beta^{0},y^{0})\in\R^{n}\times\R^{m}\times\R^{n}\times\R^{m}$, set $k=0$ and iterate.
\begin{description}
\item [Step 1.] Compute
\begin{eqnarray*}
u^{k+1}&=&\mathop{\mbox{argmin}}_{u\in\R^{m}}\left\{\tilde{\mathcal{L}}_{\zeta}(u,v^{k},w^{k};\beta^{k},y^{k})\right\}\\
&=&(I_{m}+XX^{T})^{-1}(w^{k}-\zeta^{-1}y^{k}+X(-v^{k}+\zeta^{-1}\beta^{k})-b),\\
(v^{k+1},w^{k+1})&=&\mathop{\mbox{argmin}}_{v\in\R^{n},w\in\R^{m}}\left\{\tilde{\mathcal{L}}_{\zeta}(u^{k+1},v,w;\beta^{k},y^{k})\right\}\\
&=&\left(\mbox{Prox}_{\zeta^{-1}(\lambda \pone)^{*}}(\zeta^{-1}\beta^{k}-X^{T}u^{k+1}),\mbox{Prox}_{\zeta^{-1}\delta_{B}}(\zeta^{-1}y^{k}+u^{k+1})\right).
\end{eqnarray*}
\item [Step 2.] Update
\begin{equation*}
\beta^{k+1}=\beta^{k}-\rho\zeta(X^{T}u^{k+1}+v^{k+1}),\
y^{k+1}=y^{k}-\rho\zeta(-u^{k+1}+w^{k+1}).
\end{equation*}
If the prescribed stopping
criterion is satisfied, terminate; otherwise
return to Step 1 with $k=k+1$.
\end{description}}}
\bigskip

In the pADMM and dADMM, we set the parameter $\rho=1.618$ and solve the linear system
in Step 1 of pADMM and dADMM by using the Sherman-Morrison-Woodbury formula \cite{Golub1996}
if it is necessary, i.e.,
\begin{eqnarray*}
\left(I_{n}+X^{T}X\right)^{-1}&=&I_{m}-X\left(I_{m}+XX^{T}\right)^{-1}X^{T},\\
\left(I_{m}+XX^{T}\right)^{-1}&=&I_{n}-X^{T}\left(I_{n}+X^{T}X\right)^{-1}X.
\end{eqnarray*}
Depending on the dimension $n,m$ of the problem, we either
solve the linear system with coefficient matrix $I_m + XX^T$ (or $I_n + X^{T}X$) by Cholesky factorization or
by an iterative solver such as the preconditioned conjugate gradient (PCG) method.
We should mention that when the latter approach is used, the linear system only
needs to be solved to a sufficient level of accuracy that depend on the progress
of the algorithm without sacrificing the convergence of the ADMMs.
For the details,
we refer the reader to \cite{ChenST2017}.

\subsubsection{Stopping criteria}
In order to measure the accuracy of an approximate optimal solution $\beta$, we
use the relative duality gap
defined by
\begin{eqnarray*}
\eta_{G}:=\frac{|\mbox{pobj}-\mbox{dobj}|}{1+|\mbox{pobj}|+|\mbox{dobj}|},
\end{eqnarray*}
where
$
\mbox{pobj}:=\|X\beta-b\|+\lambda \pone(\beta),\; \mbox{dobj}:=-\langle u,b\rangle
$
are the primal and dual objective values, respectively.
We also adopt the relative KKT residual
\begin{eqnarray*}
\eta_{\rm kkt}:=\frac{\left\|\beta-\mbox{Prox}_{\lambda \pone}\left(\beta-\frac{X^{T}(X\beta-b)}{\|X\beta-b\|}\right)\right\|}{1+\|\beta\|+\frac{\left\|X^{T}(X\beta-b)\right\|}{\|X\beta-b\|}}
\end{eqnarray*}
to measure the accuracy of an approximate optimal solution $\beta$.
For a given tolerance, our stage I algorithm is terminated if
\begin{eqnarray}\label{stop-cond-1}
\eta_{\rm kkt}< \epsilon_{\rm kkt}  = 10^{-6},
\end{eqnarray}
or the number of iterations reaches the maximum of 200 while the ADMMs are  terminated if (\ref{stop-cond-1}) is satisfied or the number of iterations reaches the maximum of 10000. All the algorithms are stopped if they reach the pre-set maximum running time of 4 hours.

\subsubsection{Numerical results for the srLasso problem \eqref{eq:SRLasso}}

Here we compare the performance of
 different methods for solving the convex problem
 \eqref{eq:SRLasso}.
 In
\cite{Sara2017}, it adopted the R package Flare \cite{Liuhan2015} to solve the srLasso problem \eqref{eq:SRLasso}. As the algorithm in Flare is in fact the pADMM with unit steplength,
we first compare our own implementation of the pADMM with that in the Flare package. For a
fair comparison, our pADMM is also written in R.
Since the stopping criterion of the Flare package is not stated explicitly, we first run the Flare package to obtain a primal objective value and then run our pADMM, which is terminated as soon as our primal objective value is smaller than that obtained by Flare.
We note that since \eqref{eq:SRLasso} is an unconstrained optimization problem, it is
meaningful to compare the objective function values obtained by Flare and our pADMM.

We report the numerical results in Tables 1 and 2. We report the problem name (probname), the number of samples ($m$) and features ($n$), $\lambda_{c}$, the primal objective value (pobj), and the computation time (time) in the format of
``hours:minutes:seconds''. The symbol ``--'' in Table 2 means that the Flare package fails to solve the problem due to excessive memory requirement.
From Tables 1 and 2, we can observe that our pADMM is clearly faster than the Flare package. A possible cause of this difference may lie in the different strategies for
dynamically updating the parameter $\zeta$ in the practical implementations of the pADMM.
As our implementation of the pADMM is much more efficient than that in the Flare package,
in the subsequent experiments,
we will not compare the performance of our PMM algorithm with the Flare package
but with our own pADMM.

\begin{table}[!htbp]
\setlength{\belowcaptionskip}{10pt}
\parbox{.6\textwidth}{\caption{The performance of the Flare package and {our} pADMM on synthetic datasets for the srLasso problem.}}\scriptsize\centering\label{table-compare-convex-Flare-random}
\begin{tabular}{|c|c|c|c|c|c|}
\hline
\multirow{1}{*}{probname}&\multirow{2}{*}{$\lambda_{c}$}&\multicolumn{2}{c|}{pobj}&\multicolumn{2}{c|}{time}\\
\cline{3-6} m; n & & Flare & pADMM & Flare & pADMM\\\hline
exmp1 & 1.0 & 3.8876+3 & 3.5799+3 & 11:26 & 12 \\
8000;800 & 0.5 & 3.0501+3 & 1.9174+3 & 21:09 & 13\\
  & 0.1 & 1.0487+3 & 5.8738+2 & 28:42 & 16\\\hline
exmp2 & 1.0 & 2.2422+3 & 2.2419+3 & 14:09 & 19\\
8000;800 & 0.5 &1.8050+3 & 1.2811+3 & 27:18 & 11\\
 & 0.1 & 5.6150+2 & 4.6013+2 & 27:37 & 09\\\hline
exmp3 & 1.0 & 2.4758+3 & 2.4569+3& 10:05 & 07\\
 8000;400 & 0.5 & 1.9819+3 & 1.9421+3& 7:26 & 07\ \\
 & 0.1 & 1.4888+3 & 1.4438+3 & 7:14 & 05\ \\\hline
exmp4 & 1.0 & 1.1210+4 & 1.1205+4 & 29:11 & 20:16\\
 8000;4000 & 0.5 & 1.0165+4 & 1.0165+4 & 1:43:48 & 21:48\\
 & 0.1 & 7.6846+3 & 3.4069+3 & 3:11:27 & 5:12\ \ \\
\hline
\end{tabular}
\end{table}

\begin{table}[!h]
\setlength{\belowcaptionskip}{10pt}
\parbox{.6\textwidth}{\caption{The performance of the Flare package and our pADMM on UCI datasets for the srLasso problem.}}\scriptsize\centering\label{table-compare-convex-Flare-UCI}
\begin{tabular}{|c|c|c|c|c|c|}
\hline
\multirow{1}{*}{probname}&\multirow{2}{*}{$\lambda_{c}$}&\multicolumn{2}{c|}{pobj}&\multicolumn{2}{c|}{time}\\
\cline{3-6} m; n & & Flare & pADMM & Flare & pADMM\\\hline
abalone.scale.expanded7 & 1.0 & -- & 2.3852+2 & -- & 25:57 \\
4177;6435 & 0.5 & -- & 2.0312+2 & -- & 25:32\\
  & 0.1 & -- & 1.5586+2 & -- & 26:29\\\hline
mpg.scale.expanded7 & 1.0 & 2.3550+2 & 2.3544+2 & 1:00 & 04\\
392;3432 & 0.5 & 1.5856+2 & 1.5831+2 & 57 & 03\\
 & 0.1 & 7.8656+1 & 7.8616+1 & 1:06 & 03\\\hline
 space.ga.scale.expanded9 & 1.0 & 1.3113+1 & 1.3113+1 & 12:59 & 5:19\\
 3107;5005 & 0.5 & 2.2419+1 & 2.1607+1& 9:01 & 2:00\\
 & 0.1 & 1.2950+1 & 1.1999+1 & 6:13 & 3:00\\
\hline
\end{tabular}
\end{table}

Next we conduct numerical experiments to
evaluate the performance of the pADMM, dADMM and PMM.
For the numerical results, besides the results reported in Tables 1 and 2, we also report the relative KKT residual ($\eta_{\rm kkt}$), the relative duality gap ($\eta_{G}$),  the number of nonzero elements of $\beta$ (nnz), the mean square error defined by $\|\beta-\ddot{\beta}\|^{2}/n$ (MSE), and the percentage (P) of the nonzero positions of $\ddot{\beta}$ that are picked up by $\beta$. The last three results were obtained from the PMM algorithm.  In the implementation of the pADMM and dADMM, we first compute the (sparse) Cholesky decomposition
of $I_{n}+X^{T}X$ or $I_{m}+XX^{T}$ and then solve the linear system of equations
in each iteration by using the pre-computed Cholesky factor.

\begin{table}[!h]
\begin{center}
\tiny
\parbox{0.7\textwidth}{\caption{The performance of different algorithms on synthetic datasets for the srLasso problem. In the table, ``a''=PMM, ``b"=pADMM, ``c''=dADMM.}}\label{table-compare-random}
\begin{tabular}{|p{1cm}<{\centering}|p{0.2cm}<{\centering}|p{0.4cm}<{\centering}|p{0.45cm}<{\centering}|p{0.4cm}<{\centering}|p{0.4cm}<{\centering}|p{0.5cm}<{\centering}|p{0.4cm}<{\centering}|p{0.4cm}<{\centering}|p{0.9cm}<{\centering}|p{0.9cm}<{\centering}|p{0.9cm}<{\centering}|p{0.2cm}<{\centering}|p{0.3cm}<{\centering}|p{0.3cm}<{\centering}|p{0.85cm}<{\centering}|p{0.4cm}<{\centering}|}
\hline
\multirow{2}{*}{probname}&\multirow{2}{*}{$\lambda_{c}$}&\multirow{2}{*}{nnz}&\multicolumn{3}{c|}{$\eta_{\rm kkt}$}&\multicolumn{3}{c|}{$\eta_{G}$}&\multicolumn{3}{c|}{pobj}&\multicolumn{3}{c|}{time}&\multirow{2}{*}{MSE}&\multirow{2}{*}{P}\\\cline{4-15} m; n & & & a & b & c & a & b & c & a & b & c & a & b & c & & \\[2pt] \hline 
\mc{1}{|c|}{exmp1} 
	 &\multirow{2}{*}{1}	 &\multirow{2}{*}{305}&\multirow{2}{*}{8.6-7}&\multirow{2}{*}{1.0-6}&\multirow{2}{*}{9.7-7} &\multirow{2}{*}{8.2-10}&\multirow{2}{*}{3.6-6}&\multirow{2}{*}{1.4-9}	 &\multirow{2}{*}{3.3895+3}&\multirow{2}{*}{3.3895+3}&\multirow{2}{*}{3.3895+3} &\multirow{2}{*}{12}&\multirow{2}{*}{1:03}&\multirow{2}{*}{5:09} & \multirow{2}{*}{1.4882-2}& \multirow{2}{*}{$100\%$}
\\[2pt] 
\mc{1}{|c|}{8000;800} & & & & & & & & & & & & & & & & \\[2pt] 
\hline\mc{1}{|c|}{exmp2} 
	 &\multirow{2}{*}{1}	 &\multirow{2}{*}{437}&\multirow{2}{*}{4.4-7}&\multirow{2}{*}{1.0-6}&\multirow{2}{*}{1.0-6} &\multirow{2}{*}{1.8-9}&\multirow{2}{*}{1.1-6}&\multirow{2}{*}{2.9-7}	 &\multirow{2}{*}{2.1067+3}&\multirow{2}{*}{2.1067+3}&\multirow{2}{*}{2.1067+3} &\multirow{2}{*}{17}&\multirow{2}{*}{1:10}&\multirow{2}{*}{5:21} & \multirow{2}{*}{2.9740-2}& \multirow{2}{*}{$100\%$}
\\[2pt] 
\mc{1}{|c|}{8000;800} & & & & & & & & & & & & & & & & \\[2pt] 
\hline\mc{1}{|c|}{exmp3} 
	 &\multirow{2}{*}{1}	 &\multirow{2}{*}{277}&\multirow{2}{*}{4.0-7}&\multirow{2}{*}{1.0-6}&\multirow{2}{*}{1.0-6} &\multirow{2}{*}{6.8-10}&\multirow{2}{*}{1.5-6}&\multirow{2}{*}{1.1-8}	 &\multirow{2}{*}{2.2169+3}&\multirow{2}{*}{2.2169+3}&\multirow{2}{*}{2.2169+3} &\multirow{2}{*}{07}&\multirow{2}{*}{1:16}&\multirow{2}{*}{1:01} & \multirow{2}{*}{1.2955-1}& \multirow{2}{*}{$98\%$}
\\[2pt] 
\mc{1}{|c|}{8000;400} & & & & & & & & & & & & & & & & \\[2pt] 
\hline\mc{1}{|c|}{exmp4} 
	 &\multirow{2}{*}{1}	 &\multirow{2}{*}{300}&\multirow{2}{*}{7.6-7}&\multirow{2}{*}{9.9-7}&\multirow{2}{*}{1.0-6} &\multirow{2}{*}{3.1-9}&\multirow{2}{*}{1.3-6}&\multirow{2}{*}{4.1-7}	 &\multirow{2}{*}{4.5283+3}&\multirow{2}{*}{4.5283+3}&\multirow{2}{*}{4.5283+3} &\multirow{2}{*}{12}&\multirow{2}{*}{2:36}&\multirow{2}{*}{4:47} & \multirow{2}{*}{2.5358-1}& \multirow{2}{*}{$100\%$}
\\[2pt] 
\mc{1}{|c|}{8000;800} & & & & & & & & & & & & & & & & \\[2pt] 
\hline

\end{tabular}
\end{center}
\end{table}

\begin{table}[!h]
\tiny
\begin{center}
\setlength{\belowcaptionskip}{10pt}
\parbox{.7\textwidth}{\caption{The performance of different algorithms
on UCI datasets for the srLasso problem. In the table, ``a''=PMM, ``b"=pADMM, ``c''=dADMM.}}\centering\label{table-compare-UCI}
\begin{tabular}{|p{1.8cm}<{\centering}|p{0.2cm}<{\centering}|p{0.3cm}<{\centering}|p{0.45cm}<{\centering}|p{0.4cm}<{\centering}|p{0.4cm}<{\centering}|p{0.45cm}<{\centering}|p{0.4cm}<{\centering}|p{0.48cm}<{\centering}|p{0.95cm}<{\centering}|p{0.95cm}<{\centering}|p{0.95cm}<{\centering}|p{0.3cm}<{\centering}|p{0.6cm}<{\centering}|p{0.6cm}<{\centering}|}
\hline
\multirow{2}{*}{probname}&\multirow{2}{*}{$\lambda_{c}$}&\multirow{2}{*}{nnz}&\multicolumn{3}{c|}{$\eta_{\rm kkt}$}&\multicolumn{3}{c|}{$\eta_{G}$}&\multicolumn{3}{c|}{pobj}&\multicolumn{3}{c|}{time}\\\cline{4-15} m; n & & & a & b & c & a & b & c & a & b & c & a & b & c\\[2pt] \hline 
\mc{1}{|c|}{E2006.test} 
	 &\multirow{2}{*}{1}	 &\multirow{2}{*}{1}&\multirow{2}{*}{9.4-9}&\multirow{2}{*}{9.8-7}&\multirow{2}{*}{9.1-7} &\multirow{2}{*}{1.5-6}&\multirow{2}{*}{1.1-5}&\multirow{2}{*}{5.3-10}	 &\multirow{2}{*}{2.6706+1}&\multirow{2}{*}{2.6706+1}&\multirow{2}{*}{2.6706+1} &\multirow{2}{*}{10}&\multirow{2}{*}{03}&\multirow{2}{*}{03} 
\\[2pt] 
\mc{1}{|c|}{3308;150358} & & & & & & & & & & & & & &\\[2pt] 
\hline\mc{1}{|c|}{log1p.E2006.test} 
	 &\multirow{2}{*}{1}	 &\multirow{2}{*}{5}&\multirow{2}{*}{9.8-7}&\multirow{2}{*}{\emph{1.2-4}}&\multirow{2}{*}{\emph{5.2-3}} &\multirow{2}{*}{2.7-4}&\multirow{2}{*}{8.8-2}&\multirow{2}{*}{1.3-3}	 &\multirow{2}{*}{2.6046+1}&\multirow{2}{*}{2.8627+1}&\multirow{2}{*}{2.6046+1} &\multirow{2}{*}{1:06}&\multirow{2}{*}{1:10:39}&\multirow{2}{*}{49:11} 
\\[2pt]
\mc{1}{|c|}{3308;1771946} & & & & & & & & & & & & & & \\[2pt] 
\hline
\mc{1}{|c|}{E2006.train} 
	 &\multirow{2}{*}{1}	 &\multirow{2}{*}{1}&\multirow{2}{*}{3.8-7}&\multirow{2}{*}{8.0-7}&\multirow{2}{*}{9.6-7} &\multirow{2}{*}{2.2-6}&\multirow{2}{*}{3.1-6}&\multirow{2}{*}{2.1-10}	 &\multirow{2}{*}{5.4180+1}&\multirow{2}{*}{5.4180+1}&\multirow{2}{*}{5.4180+1} &\multirow{2}{*}{26}&\multirow{2}{*}{35}&\multirow{2}{*}{44} 
\\[2pt] 
\mc{1}{|c|}{16087;150358} & & & & & & & & & & & & & &\\[2pt] 
\hline
\mc{1}{|c|}{log1p.E2006.train} 
	 &\multirow{2}{*}{1}	 &\multirow{2}{*}{22}&\multirow{2}{*}{4.5-7}&\multirow{2}{*}{\emph{1.6-1}}&\multirow{2}{*}{\emph{7.6-1}} &\multirow{2}{*}{5.9-4}&\multirow{2}{*}{9.9-1}&\multirow{2}{*}{5.6-1}	 &\multirow{2}{*}{5.2032+1}&\multirow{2}{*}{6.9860+5}&\multirow{2}{*}{8.1873+2} &\multirow{2}{*}{4:27}&\multirow{2}{*}{4:00:19}&\multirow{2}{*}{4:01:26} 
\\[2pt] 
\mc{1}{|c|}{16087;4272227} & & & & & & & & & & & & & &\\[2pt] 
\hline
\mc{1}{|c|}{abalone.scale.expanded7} 
	 &\multirow{2}{*}{1}	 &\multirow{2}{*}{6}&\multirow{2}{*}{9.7-7}&\multirow{2}{*}{\emph{9.6-6}}&\multirow{2}{*}{\emph{7.5-3}} &\multirow{2}{*}{9.9-8}&\multirow{2}{*}{1.1-3}&\multirow{2}{*}{1.3-2}	 &\multirow{2}{*}{2.3562+2}&\multirow{2}{*}{2.3589+2}&\multirow{2}{*}{2.3575+2} &\multirow{2}{*}{04}&\multirow{2}{*}{20:10}&\multirow{2}{*}{17:57} 
\\[2pt] 
\mc{1}{|c|}{4177;6435} & & & & & & & & & & & & & & \\[2pt] 
\hline\mc{1}{|c|}{housing.scale.expanded7} 
	 &\multirow{2}{*}{1}	 &\multirow{2}{*}{22}&\multirow{2}{*}{7.7-7}&\multirow{2}{*}{\emph{2.7-6}}&\multirow{2}{*}{\emph{3.7-4}} &\multirow{2}{*}{4.1-6}&\multirow{2}{*}{1.2-3}&\multirow{2}{*}{1.5-3}	 &\multirow{2}{*}{2.6957+2}&\multirow{2}{*}{2.6989+2}&\multirow{2}{*}{2.6957+2} &\multirow{2}{*}{07}&\multirow{2}{*}{40:35}&\multirow{2}{*}{22:48} 
\\[2pt] 
\mc{1}{|c|}{506;77520} & & & & & & & & & & & & & & \\[2pt] 
\hline\mc{1}{|c|}{mpg.scale.expanded7} 
	 &\multirow{2}{*}{1}	 &\multirow{2}{*}{5}&\multirow{2}{*}{9.6-7}&\multirow{2}{*}{1.0-6}&\multirow{2}{*}{1.0-6} &\multirow{2}{*}{1.6-8}&\multirow{2}{*}{4.0-5}&\multirow{2}{*}{1.2-6}	 &\multirow{2}{*}{2.1320+2}&\multirow{2}{*}{2.1321+2}&\multirow{2}{*}{2.1320+2} &\multirow{2}{*}{01}&\multirow{2}{*}{18}&\multirow{2}{*}{26} 
\\[2pt] 
\mc{1}{|c|}{392;3432} & & & & & & & & & & & & & & \\[2pt] 
\hline\mc{1}{|c|}{space.ga.scale.expanded9} 
	 &\multirow{2}{*}{1}	 &\multirow{2}{*}{4}&\multirow{2}{*}{5.5-7}&\multirow{2}{*}{1.0-6}&\multirow{2}{*}{9.6-7} &\multirow{2}{*}{8.2-8}&\multirow{2}{*}{2.7-4}&\multirow{2}{*}{1.1-7}	 &\multirow{2}{*}{1.3111+1}&\multirow{2}{*}{1.3115+1}&\multirow{2}{*}{1.3111+1} &\multirow{2}{*}{03}&\multirow{2}{*}{1:57}&\multirow{2}{*}{56} 
\\[2pt] 
\mc{1}{|c|}{3107;5005} & & & & & & & & & & & & & & \\[2pt] 
\hline
\mc{1}{|c|}{pyrim.scale.expanded5} 
	 &\multirow{2}{*}{1}	 &\multirow{2}{*}{23}&\multirow{2}{*}{8.0-7}&\multirow{2}{*}{\emph{5.6-6}}&\multirow{2}{*}{\emph{8.5-2}} &\multirow{2}{*}{1.8-4}&\multirow{2}{*}{7.4-2}&\multirow{2}{*}{2.1-2}	 &\multirow{2}{*}{3.3790+0}&\multirow{2}{*}{3.7206+0}&\multirow{2}{*}{3.7958+0} &\multirow{2}{*}{07}&\multirow{2}{*}{15:11}&\multirow{2}{*}{9:39} 
\\[2pt] 
\mc{1}{|c|}{74;201376} & & & & & & & & & & & & & & \\[2pt] 
\hline
\mc{1}{|c|}{bodyfat.scale.expanded7} 
	 &\multirow{2}{*}{1}	 &\multirow{2}{*}{2}&\multirow{2}{*}{7.0-7}&\multirow{2}{*}{\emph{1.9-6}}&\multirow{2}{*}{1.0-6} &\multirow{2}{*}{1.3-5}&\multirow{2}{*}{2.1-2}&\multirow{2}{*}{2.5-8}	 &\multirow{2}{*}{4.5326+0}&\multirow{2}{*}{4.6458+0}&\multirow{2}{*}{4.5326+0} &\multirow{2}{*}{07}&\multirow{2}{*}{30:53}&\multirow{2}{*}{6:18} 
\\[2pt] 
\mc{1}{|c|}{252;116280} & & & & & & & & & & & & & & \\[2pt] 
\hline

\end{tabular}
\end{center}
\end{table}

Tables 3 and 4 show the performance of the three algorithms. For the synthetic datasets, the pADMM is more efficient than the dADMM in almost all cases. Furthermore, we can see that our PMM algorithm can solve all the problems to the required accuracy. The PMM algorithm not only takes much less time than the pADMM or dADMM does but also obtains
more accurate solutions (in terms of $\eta_{\rm kkt}$) in almost all cases.

\subsection{Numerical experiments for the square-root regression problems with nonconvex regularizers}

In this section, we compare the performance of the ADMM and our PMM algorithm for solving the nonconvex square-root regression problem (\ref{eq:prob}). The relative KKT residual
\begin{eqnarray}\label{stop-cond-2}
\tilde{\eta}_{\rm kkt}:=\frac{\left\|\beta-\mbox{Prox}_{\lambda \pone-\ptwo}\left(\beta-\frac{X^{T}(X\beta-b)}{\|X\beta-b\|}\right)\right\|}{1+\|\beta\|+\frac{\|X^{T}(X\beta-b)\|}{\|X\beta-b\|}}
\end{eqnarray}
is adopted to measure the accuracy of an approximate optimal solution $\beta$.
In our PMM algorithm, stage I is implemented to generate an initial point for stage I\!I and is stopped if $\eta_{\rm kkt}<10^{-4}$. The tested
algorithms will be terminated if $\tilde{\eta}_{\rm kkt}<\tilde{\epsilon}_{\rm kkt}=10^{-6}$.
In addition, the algorithms are also stopped when they reach the pre-set maximum number of iterations (200 for stage I\!I of the PMM and 10000 for the ADMM) or the pre-set maximum running time of 4 hours. For each
synthetic data set, the models are fitted on the training data set and the validation data set
is used to select the regularization parameter $\lambda_{c}$. For each UCI data set, we adopt a tenfold cross validation to find the regularization parameter. The PMM algorithm is used to perform the cross validation.

\subsubsection{The ADMM for the problem (\ref{eq:prob})}

To describe the ADMM implemented (which is not guaranteed to  converge though due to the nonconvexity) for solving the nonconvex square-root regression problem (\ref{eq:prob}), we first reformulate it to the following constrained problem:
\begin{eqnarray}\label{primal-problem-nonconvex-1}
\min_{\beta,z\in\R^{n},y\in\R^{m}}\left\{\|y\|+ \lambda \pone(z)-\ptwo(z)\,\Big|\,X\beta-y=b,\beta-z=0\right\}.
\end{eqnarray}
For $\zeta>0$, the augmented Lagrangian function of (\ref{primal-problem-nonconvex-1}) can be written as
\begin{eqnarray*}
L_{\zeta}(\beta,y,z;u,v)&:=&\|y\|+\lambda \pone(z)-\ptwo(z)+\langle u,X\beta-y-b\rangle+\frac{\zeta}{2}\|X\beta-y-b\|^{2}\\
&&+\langle v,\beta-z\rangle+\frac{\zeta}{2}\|\beta-z\|^{2}.
\end{eqnarray*}
The template of the ADMM for solving the problem (\ref{eq:prob})
is given as the following form.

\bigskip
\ni\fbox{\parbox{\textwidth}{\noindent{\bf Algorithm ADMM for the problem (\ref{eq:prob}).}
Let $\zeta>0$ be a given parameter. Choose $(y^{0},z^{0},u^{0},v^{0})\in\R^{m}\times\R^{n}\times\R^{m}\times\R^{n}$, set $k=0$ and iterate.
\begin{description}
\item [Step 1.] Compute
\begin{align*}
\beta^{k+1}&=\mathop{\mbox{argmin}}_{\beta\in\R^{n}}\left\{L_{\zeta}(\beta,y^{k},z^{k};u^{k},v^{k})\right\}\\
&=(I_n+X^{T}X)^{-1}(z^{k}-\zeta^{-1}v^{k}+X^{T}(y^{k}+b-\zeta^{-1}u^{k})),\\
(y^{k+1},z^{k+1})&=\mathop{\mbox{argmin}}_{y\in\R^{m},z\in\R^{n}}\left\{L_{\zeta}(\beta^{k+1},y,z;u^{k},v^{k})\right\}\\
&=\left(\mbox{Prox}_{\zeta^{-1}\|\cdot\|}(X\beta^{k+1}-b+\zeta^{-1}u^{k}),\mbox{Prox}_{\zeta^{-1}(\lambda \pone-\ptwo)}(\beta^{k+1}+\zeta^{-1}v^{k})\right).
\end{align*}
\item [Step 2.] Update
\begin{equation*}
u^{k+1}=u^{k}+\zeta(X\beta^{k+1}-y^{k+1}-b),\ v^{k+1}=v^{k}+\zeta(\beta^{k+1}-z^{k+1}).
\end{equation*}
If the prescribed stopping
criterion is satisfied, terminate; otherwise
 return to Step 1 with $k=k+1$.
\end{description}}}
\bigskip

\subsubsection{Numerical experiments for the square-root regression problems with SCAD regularizations}

The SCAD regularization involves a concave function $p_\lambda$, proposed in
\cite{Fan2001}, that has the following properties: $p_{\lambda}(0)=0$ and for $|t|>0$,
\begin{eqnarray*}
p'_{\lambda}(|t|)=\left\{\begin{array}{lll}\lambda,&\mbox{if}\
|t|\leq\lambda,\\
\frac{(a_{s}\lambda-|t|)_{+}}{a_{s}-1},&\mbox{otherwise},\end{array}\right.
\end{eqnarray*}
for some given parameter $a_{s}>2$. In the above,
$(a_{s}\lambda-|t|)_{+}$ denotes the positive part of
$a_{s}\lambda-|t|$. We can reformulate the expression of the SCAD
regularization function as $\lambda \pone(\beta)-\ptwo(\beta)$
with $\pone(\beta)=\|\beta\|_{1}$ and
\begin{eqnarray*}
\ptwo(\beta)=\sum_{i=1}^{n}\ptwo^{\rm scad}(\beta_{i};a_{s},\lambda),\
\ptwo^{\rm scad}(t;a_{s},\lambda)=\left\{\begin{array}{ll}0,&\mbox{if}\quad
|t|\leq \lambda,\\
\frac{(|t|-\lambda)^2}{2(a_{s}-1)},&\mbox{if}\quad \lambda\leq
|t|\leq a_{s}\lambda,\\
\lambda |t|-\frac{a_{s}+1}{2}\lambda^2,&\mbox{if}\quad
|t|>a_{s}\lambda.
\end{array}\right.
\end{eqnarray*}
The function $\ptwo(\beta)$ is continuously differentiable with
\begin{eqnarray*}
\frac{\partial \ptwo(\beta)}{\partial\beta_{i}}=\left\{\begin{array}{ll}0,&\mbox{if}\ |\beta_{i}|\leq\lambda,\\
\frac{\mbox{sign}(\beta_{i})(|\beta_{i}|-\lambda)}{a_{s}-1},&\mbox{if}\ \lambda<|\beta_{i}|\leq a_{s}\lambda,\\
\lambda\mbox{sign}(\beta_{i}),&\mbox{if}\ |\beta_{i}|>a_{s}\lambda.\end{array}\right.
\end{eqnarray*}
We can see that the SCAD regularization function associated with $\beta_{i}$ is increasing and concave in
$[0,+\infty)$. It has been shown in \cite{Fan2001} that
the SCAD regularization usually performs better than the classical $\ell_1$ regularization
in
selecting significant variables without creating excessive biases.

The performance of the PMM algorithm and ADMM for the SCAD regularization with
$a_{s}=3.7$ are listed in Tables 5 and 6. We can see that in most cases, the PMM algorithm is not only much more efficient than the ADMM, but it can also obtain
better objective function values. Although the objective value of the ADMM is less than that of the PMM algorithm in the housing.scale.expanded7 data set, the solution of the PMM algorithm is more sparse than that of the ADMM with nnz being 62 versus 68777.

%
%
%

\begin{table}[!h]
\setlength{\belowcaptionskip}{10pt}
\parbox{.6\textwidth}{\caption{The performance of the ADMM and PMM on synthetic datasets for the SCAD regularization. In the table, ``a"=PMM, ``b''=ADMM.}}\scriptsize\centering\label{table-compare-SCAD-random}
\begin{tabular}{|c|c|c|c|c|c|c|c|c|c|c|c|}
\hline
\multirow{2}{*}{probname}&\multirow{2}{*}{$\lambda_{c}$}&\multirow{2}{*}{nnz}&\multicolumn{2}{c|}{$\eta_{\rm kkt}$}&\multicolumn{2}{c|}{pobj}&\multicolumn{2}{c|}{time}&\multirow{2}{*}{MSE}&\multirow{2}{*}{P}\\\cline{4-9} m; n & &  & a & b & a & b & a & b & & \\[2pt] \hline 
\mc{1}{|c|}{exmp1} 
	 &\multirow{2}{*}{0.145}	 &\multirow{2}{*}{460}&\multirow{2}{*}{3.9-7}&\multirow{2}{*}{\emph{5.9-1}} &\multirow{2}{*}{5.9368+2}&\multirow{2}{*}{5.9392+2} &\multirow{2}{*}{20}&\multirow{2}{*}{3:39}& \multirow{2}{*}{9.1948-4}& \multirow{2}{*}{$100\%$}
\\[2pt] 
\mc{1}{|c|}{8000;800} & & & & & & & & & & \\[2pt] 
\hline\mc{1}{|c|}{exmp2} 
	 &\multirow{2}{*}{0.087}	 &\multirow{2}{*}{616}&\multirow{2}{*}{5.9-7}&\multirow{2}{*}{\emph{1.0-1}} &\multirow{2}{*}{4.0760+2}&\multirow{2}{*}{4.0777+2} &\multirow{2}{*}{28}&\multirow{2}{*}{3:33}& \multirow{2}{*}{1.3380-3}& \multirow{2}{*}{$100\%$}
\\[2pt] 
\mc{1}{|c|}{8000;800} & & & & & & & & & & \\[2pt] 
\hline\mc{1}{|c|}{exmp3} 
	 &\multirow{2}{*}{0.230}	 &\multirow{2}{*}{293}&\multirow{2}{*}{7.8-7}&\multirow{2}{*}{\emph{2.7-1}} &\multirow{2}{*}{1.5486+3}&\multirow{2}{*}{1.5529+3} &\multirow{2}{*}{10}&\multirow{2}{*}{2:02}& \multirow{2}{*}{9.8282-2}& \multirow{2}{*}{$100\%$}
\\[2pt] 
\mc{1}{|c|}{8000;400} & & & & & & & & & & \\[2pt] 
\hline\mc{1}{|c|}{exmp4} 
	 &\multirow{2}{*}{0.184}	 &\multirow{2}{*}{451}&\multirow{2}{*}{6.4-7}&\multirow{2}{*}{\emph{7.0-1}} &\multirow{2}{*}{8.4087+2}&\multirow{2}{*}{8.4154+2} &\multirow{2}{*}{15}&\multirow{2}{*}{4:45}& \multirow{2}{*}{5.7208-4}& \multirow{2}{*}{$100\%$}
\\[2pt] 
\mc{1}{|c|}{800;4000} & & & & & & & & & & \\[2pt] 
\hline

\end{tabular}
\end{table}

\begin{table}[!h]
\setlength{\belowcaptionskip}{10pt}
\parbox{.6\textwidth}{\caption{The performance of the ADMM and PMM on UCI datasets for the SCAD regularization. In the table,``a"=PMM, ``b''=ADMM.}}\scriptsize\centering\label{table-compare-SCAD-UCI}
\begin{tabular}{|c|c|c|c|c|c|c|c|c|}
\hline
\multirow{2}{*}{probname}&\multirow{2}{*}{$\lambda_{c}$}&\multirow{2}{*}{nnz}&\multicolumn{2}{c|}{$\eta_{\rm kkt}$}&\multicolumn{2}{c|}{pobj}&\multicolumn{2}{c|}{time}\\\cline{4-9} m; n & & & a & b & a & b & a & b \\[2pt] \hline 
\mc{1}{|c|}{E2006.test} 
	 &\multirow{2}{*}{0.071}	 &\multirow{2}{*}{1}&\multirow{2}{*}{2.2-8}&\multirow{2}{*}{9.0-7} &\multirow{2}{*}{2.2165+1}&\multirow{2}{*}{2.2165+1} &\multirow{2}{*}{08}&\multirow{2}{*}{12:51} 
\\[2pt] 
\mc{1}{|c|}{3308;150358} & & & & & & & &\\[2pt] 
\hline\mc{1}{|c|}{log1p.E2006.test} 
	 &\multirow{2}{*}{0.257}	 &\multirow{2}{*}{207}&\multirow{2}{*}{2.1-7}&\multirow{2}{*}{\emph{5.9-3}} &\multirow{2}{*}{2.1613+1}&\multirow{2}{*}{2.1366+2} &\multirow{2}{*}{3:50}&\multirow{2}{*}{2:36:14} 
\\[2pt] 
\mc{1}{|c|}{3308;1771946} & & & & & & & &\\[2pt] 
\hline
\mc{1}{|c|}{E2006.train} 
	 &\multirow{2}{*}{0.021}	 &\multirow{2}{*}{1}&\multirow{2}{*}{5.1-7}&\multirow{2}{*}{\emph{9.8-1}} &\multirow{2}{*}{4.8922+1}&\multirow{2}{*}{4.9442+1} &\multirow{2}{*}{12}&\multirow{2}{*}{3:10:28} 
\\[2pt] 
\mc{1}{|c|}{16087;150358} & & & & & & & &\\[2pt] 
\hline\mc{1}{|c|}{log1p.E2006.train} 
	 &\multirow{2}{*}{0.562}	 &\multirow{2}{*}{65}&\multirow{2}{*}{2.7-7}&\multirow{2}{*}{\emph{9.6-1}} &\multirow{2}{*}{4.9516+1}&\multirow{2}{*}{2.8862+2} &\multirow{2}{*}{6:49}&\multirow{2}{*}{4:02:03} 
\\[2pt] 
\mc{1}{|c|}{16087;4272227} & & & & & & & &\\[2pt] 
\hline
\mc{1}{|c|}{abalone.scale.expanded7} 
	 &\multirow{2}{*}{0.011}	 &\multirow{2}{*}{49}&\multirow{2}{*}{9.9-7}&\multirow{2}{*}{\emph{6.9-1}} &\multirow{2}{*}{1.3292+2}&\multirow{2}{*}{1.3864+2} &\multirow{2}{*}{12}&\multirow{2}{*}{21:41} 
\\[2pt] 
\mc{1}{|c|}{4177;6435} & & & & & & & &\\[2pt] 
\hline\mc{1}{|c|}{housing.scale.expanded7} 
	 &\multirow{2}{*}{0.070}	 &\multirow{2}{*}{62}&\multirow{2}{*}{2.1-7}&\multirow{2}{*}{\emph{5.0-1}} &\multirow{2}{*}{6.1449+1}&\multirow{2}{*}{5.7203+1} &\multirow{2}{*}{25}&\multirow{2}{*}{30:37} 
\\[2pt] 
\mc{1}{|c|}{506;77520} & & & & & & & &\\[2pt] 
\hline\mc{1}{|c|}{mpg.scale.expanded7} 
	 &\multirow{2}{*}{0.107}	 &\multirow{2}{*}{27}&\multirow{2}{*}{3.1-9}&\multirow{2}{*}{\emph{4.9-1}} &\multirow{2}{*}{5.5558+1}&\multirow{2}{*}{5.9918+1} &\multirow{2}{*}{01}&\multirow{2}{*}{1:38} 
\\[2pt] 
\mc{1}{|c|}{392;3432} & & & & & & & &\\[2pt] 
\hline\mc{1}{|c|}{space.ga.scale.expanded9} 
	 &\multirow{2}{*}{0.043}	 &\multirow{2}{*}{16}&\multirow{2}{*}{1.9-7}&\multirow{2}{*}{\emph{4.4-1}} &\multirow{2}{*}{6.9072+0}&\multirow{2}{*}{9.0447+0} &\multirow{2}{*}{03}&\multirow{2}{*}{12:50} 
\\[2pt] 
\mc{1}{|c|}{3107;5005} & & & & & & & &\\[2pt] 
\hline
\mc{1}{|c|}{pyrim.scale.expanded5} 
	 &\multirow{2}{*}{0.109}	 &\multirow{2}{*}{70}&\multirow{2}{*}{1.4-7}&\multirow{2}{*}{\emph{4.3-3}} &\multirow{2}{*}{6.8301-1}&\multirow{2}{*}{7.2608-1} &\multirow{2}{*}{13}&\multirow{2}{*}{21:26} 
\\[2pt] 
\mc{1}{|c|}{74;201376} & & & & & & & &\\[2pt] 
\hline
\mc{1}{|c|}{bodyfat.scale.expanded7} 
	 &\multirow{2}{*}{0.201}	 &\multirow{2}{*}{2}&\multirow{2}{*}{3.9-8}&\multirow{2}{*}{\emph{7.6-2}} &\multirow{2}{*}{9.4125-1}&\multirow{2}{*}{9.5136-1} &\multirow{2}{*}{06}&\multirow{2}{*}{25:51} 
\\[2pt] 
\mc{1}{|c|}{252;116280} & & & & & & & &\\[2pt] 
\hline

\end{tabular}
\end{table}

\subsubsection{Numerical experiments for the square-root regression problems with  MCP regularizations}

In this subsection, we consider the regularization by a minimax concave penalty (MCP) function \cite{Zhang2010}.
For two positive parameters $a_{m}>2$ and $\lambda$, the MCP regularization can be defined as $\lambda \pone(\beta)-\ptwo(\beta)$ with $\pone(\beta)=2\|\beta\|_{1}$ and
\begin{eqnarray*}
\ptwo(\beta)=\sum_{i=1}^{n}\ptwo^{\rm mcp}(\beta_{i};a_{m},\lambda),\
\ptwo^{\rm mcp}(t;a_{m},\lambda)=\left\{\begin{array}{ll}
\frac{t^{2}}{a_{m}},& \mbox{if}\ |t|\leq a_{m}\lambda,\\[5pt]
2\lambda|t|-a_{m}\lambda^2,&\mbox{if}\ |t|>a_{m}\lambda.\end{array}\right.
\end{eqnarray*}
The function $\ptwo(\beta)$ is continuously differentiable with its derivative given by
\begin{eqnarray*}
\frac{\partial \ptwo(\beta)}{\partial\beta_{i}}=\left\{\begin{array}{ll}\frac{2\beta_{i}}{a_{m}},&\mbox{if}\ |\beta_{i}|\leq a_{m}\lambda,\\[5pt]
2\lambda\mbox{sign}(\beta_{i}),&\mbox{if}\ |\beta_{i}|>a_{m}\lambda.\end{array}\right.
\end{eqnarray*}
We evaluate  the performance of our PMM algorithm on the same
set of problems as in the
last subsection with the MCP regularization. The numerical results are presented in Tables 7 and 8. In this case, we set the parameter $a_{m}=3.7$.

\begin{table}[!h]
\setlength{\belowcaptionskip}{10pt}
\parbox{.6\textwidth}{\caption{The performance of the ADMM and PMM on synthetic datasets for the MCP regularization. In the table, ``a"=PMM, ``b''=ADMM.}}\scriptsize\centering\label{table-compare-MCP-random}
\begin{tabular}{|c|c|c|c|c|c|c|c|c|c|c|}
\hline
\multirow{2}{*}{probname}&\multirow{2}{*}{$\lambda_{c}$}&\multirow{2}{*}{nnz}&\multicolumn{2}{c|}{$\eta_{\rm kkt}$}&\multicolumn{2}{c|}{pobj}&\multicolumn{2}{c|}{time}&\multirow{2}{*}{MSE}&\multirow{2}{*}{P}\\\cline{4-9} m; n & &  & a & b & a & b & a & b & & \\[2pt] \hline 
\mc{1}{|c|}{exmp1} 
	 &\multirow{2}{*}{0.209}	 &\multirow{2}{*}{380}&\multirow{2}{*}{5.2-8}&\multirow{2}{*}{\emph{1.9-2}} &\multirow{2}{*}{5.5695+2}&\multirow{2}{*}{5.6091+2} &\multirow{2}{*}{29}&\multirow{2}{*}{3:33}& \multirow{2}{*}{7.0382-4}& \multirow{2}{*}{$100\%$}
\\[2pt] 
\mc{1}{|c|}{8000;800} & & & & & & & & & & \\[2pt] 
\hline\mc{1}{|c|}{exmp2} 
	 &\multirow{2}{*}{0.151}	 &\multirow{2}{*}{535}&\multirow{2}{*}{2.7-7}&\multirow{2}{*}{\emph{1.3-1}} &\multirow{2}{*}{4.5225+2}&\multirow{2}{*}{4.5414+2} &\multirow{2}{*}{38}&\multirow{2}{*}{3:30}& \multirow{2}{*}{9.8362-4}& \multirow{2}{*}{$100\%$}
\\[2pt] 
\mc{1}{|c|}{8000;800} & & & & & & & & & & \\[2pt] 
\hline\mc{1}{|c|}{exmp3} 
	 &\multirow{2}{*}{0.081}	 &\multirow{2}{*}{267}&\multirow{2}{*}{9.3-7}&\multirow{2}{*}{\emph{1.5-1}} &\multirow{2}{*}{1.3590+3}&\multirow{2}{*}{1.3617+3} &\multirow{2}{*}{1:11}&\multirow{2}{*}{2:01}& \multirow{2}{*}{1.5072-1}& \multirow{2}{*}{$98\%$}
\\[2pt] 
\mc{1}{|c|}{8000;400} & & & & & & & & & & \\[2pt] 
\hline\mc{1}{|c|}{exmp4} 
	 &\multirow{2}{*}{0.321}	 &\multirow{2}{*}{334}&\multirow{2}{*}{7.7-7}&\multirow{2}{*}{\emph{3.5-2}} &\multirow{2}{*}{9.6711+2}&\multirow{2}{*}{9.7146+2} &\multirow{2}{*}{14}&\multirow{2}{*}{4:45}& \multirow{2}{*}{4.4676-4}& \multirow{2}{*}{$100\%$}
\\[2pt] 
\mc{1}{|c|}{800;4000} & & & & & & & & & & \\[2pt] 
\hline

\end{tabular}
\end{table}

\begin{table}[!h]
\setlength{\belowcaptionskip}{10pt}
\parbox{.6\textwidth}{\caption{The performance of the ADMM and PMM on UCI datasets for the MCP regularization. In the table, ``a"=PMM, ``b''=ADMM.}}\scriptsize\centering\label{table-compare-MCP-UCI}
\begin{tabular}{|c|c|c|c|c|c|c|c|c|}
\hline
\multirow{2}{*}{probname}&\multirow{2}{*}{$\lambda_{c}$}&\multirow{2}{*}{nnz}&\multicolumn{2}{c|}{$\eta_{\rm kkt}$}&\multicolumn{2}{c|}{pobj}&\multicolumn{2}{c|}{time}\\\cline{4-9} m; n & & & a & b & a & b & a & b \\[2pt] \hline 
\mc{1}{|c|}{E2006.test} 
	 &\multirow{2}{*}{0.090}	 &\multirow{2}{*}{1}&\multirow{2}{*}{2.4-8} &\multirow{2}{*}{9.3-7} &\multirow{2}{*}{2.2077+1}&\multirow{2}{*}{2.2077+1} &\multirow{2}{*}{07}&\multirow{2}{*}{07} 
\\[2pt] 
\mc{1}{|c|}{3308;150358} & & & & & & & &\\[2pt] 
\hline\mc{1}{|c|}{log1p.E2006.test} 
	 &\multirow{2}{*}{0.261}	 &\multirow{2}{*}{187}&\multirow{2}{*}{8.2-7}&\multirow{2}{*}{\emph{2.2-3}} &\multirow{2}{*}{2.1455+1}&\multirow{2}{*}{3.6500+1} &\multirow{2}{*}{4:09}&\multirow{2}{*}{2:20:46} 
\\[2pt] 
\mc{1}{|c|}{3308;1771946} & & & & & & & &\\[2pt] 
\hline
\mc{1}{|c|}{E2006.train} 
	 &\multirow{2}{*}{0.028}	 &\multirow{2}{*}{1}&\multirow{2}{*}{6.8-7}&\multirow{2}{*}{\emph{1.1-6}} &\multirow{2}{*}{4.8914+1}&\multirow{2}{*}{4.9256+1} &\multirow{2}{*}{20}&\multirow{2}{*}{3:12:39} 
\\[2pt] 
\mc{1}{|c|}{16087;150358} & & & & & & & &\\[2pt] 
\hline\mc{1}{|c|}{log1p.E2006.train} 
	 &\multirow{2}{*}{0.541}	 &\multirow{2}{*}{67}&\multirow{2}{*}{1.2-7}&\multirow{2}{*}{\emph{1.4-2}} &\multirow{2}{*}{4.9200+1}&\multirow{2}{*}{1.6911+2} &\multirow{2}{*}{6:49}&\multirow{2}{*}{4:02:05} 
\\[2pt] 
\mc{1}{|c|}{16087;4272227} & & & & & & & &\\[2pt] 
\hline
\mc{1}{|c|}{abalone.scale.expanded7} 
	 &\multirow{2}{*}{0.012}	 &\multirow{2}{*}{55}&\multirow{2}{*}{7.1-7}&\multirow{2}{*}{\emph{1.6-5}} &\multirow{2}{*}{1.3271+2}&\multirow{2}{*}{1.2693+2} &\multirow{2}{*}{09}&\multirow{2}{*}{21:32} 
\\[2pt] 
\mc{1}{|c|}{4177;6435} & & & & & & & &\\[2pt] 
\hline
\mc{1}{|c|}{housing.scale.expanded7} 
	 &\multirow{2}{*}{0.282}	 &\multirow{2}{*}{22}&\multirow{2}{*}{6.9-7}&\multirow{2}{*}{\emph{4.1-2}} &\multirow{2}{*}{1.1022+2}&\multirow{2}{*}{7.2573+2} &\multirow{2}{*}{26}&\multirow{2}{*}{30:09} 
\\[2pt] 
\mc{1}{|c|}{506;77520} & & & & & & & &\\[2pt] 
\hline\mc{1}{|c|}{mpg.scale.expanded7} 
	 &\multirow{2}{*}{0.102}	 &\multirow{2}{*}{23}&\multirow{2}{*}{7.3-7}&\multirow{2}{*}{\emph{2.1-2}} &\multirow{2}{*}{5.0964+1}&\multirow{2}{*}{5.9492+1} &\multirow{2}{*}{01}&\multirow{2}{*}{1:36} 
\\[2pt] 
\mc{1}{|c|}{392;3432} & & & & & & & &\\[2pt] 
\hline\mc{1}{|c|}{space.ga.scale.expanded9} 
	 &\multirow{2}{*}{0.046}	 &\multirow{2}{*}{15}&\multirow{2}{*}{5.9-7}&\multirow{2}{*}{\emph{3.3-2}} &\multirow{2}{*}{6.6399+0}&\multirow{2}{*}{7.8456+0} &\multirow{2}{*}{05}&\multirow{2}{*}{12:46} 
\\[2pt] 
\mc{1}{|c|}{3107;5005} & & & & & & & &\\[2pt] 
\hline
\mc{1}{|c|}{pyrim.scale.expanded5} 
	 &\multirow{2}{*}{0.221}	 &\multirow{2}{*}{43}&\multirow{2}{*}{9.9-7}&\multirow{2}{*}{\emph{7.0-3}} &\multirow{2}{*}{1.1428+0}&\multirow{2}{*}{4.6112+0} &\multirow{2}{*}{18}&\multirow{2}{*}{19:36} 
\\[2pt] 
\mc{1}{|c|}{74;201376} & & & & & & & &\\[2pt] 
\hline
\mc{1}{|c|}{bodyfat.scale.expanded7} 
	 &\multirow{2}{*}{0.183}	 &\multirow{2}{*}{2}&\multirow{2}{*}{2.8-7}&\multirow{2}{*}{\emph{5.3-6}} &\multirow{2}{*}{5.7347-1}&\multirow{2}{*}{5.8278-1} &\multirow{2}{*}{06}&\multirow{2}{*}{25:11} 
\\[2pt] 
\mc{1}{|c|}{252;116280} & & & & & & & &\\[2pt] 
\hline

\end{tabular}
\end{table}

From the numerical results, one can see the efficiency and power of our SSN method based
PMM algorithm. Note that though for the  abalone.scale.expanded7 data set  the objective value obtained by the ADMM is less than that obtained  the PMM algorithm, the solution obtained by  the PMM algorithm is more sparse than that by the ADMM with nnz being 55 versus 1039.
Overall,
our PMM algorithm is clearly more efficient and accurate than the ADMM on the
tested datasets.

We have mentioned in the introduction that the scaled Lasso problem is equivalent to the srLasso problem \eqref{eq:SRLasso}. However, in order to solve the scaled Lasso problem, we have to call an algorithm several times to solve the standard Lasso subproblems.
However, by  handling the srLasso problem \eqref{eq:SRLasso} directly, our algorithm is as fast as the highly efficient algorithm, LassoNAL \cite{LiSunToh2018},
for solving a single standard Lasso problem.

\section{Conclusion}
\label{sec:Conclusion}
In this paper, we proposed a two stage PMM algorithm
to solve the square-root regression problems with nonconvex regularizations.
We are able to achieve impressive computational efficiency for our algorithm
by designing an innovative proximal majorization framework
 for the convex subproblem arising in each PMM iteration so that it can be solved
 via its dual by
the SSN method. We presented the oracle property of the problem in stage I and analyzed the convergence of the PMM algorithm with its subproblems solved inexactly. Extensive numerical experiments have
demonstrated the efficiency of our PMM algorithm
when compared to other natural alternative algorithms such as the ADMM based algorithms
in solving the problem of interest.

From the superior performance of our algorithm, it is natural for us
to consider applying a similar proximal majorization-minimization
algorithmic framework to design efficient
algorithms to solve
other square-root regression problems with structured sparsity
requirements such a group sparsity in the regression coefficients \cite{BuneaLS}.
We leave such an investigation as a future research topic.

\section{Appendix}
In this appendix, we first provide the proofs for Lemma 3.1, 3.3 and  3.4. Based on these results, we then give the proof for Theorem 3.1.
\\

\noindent\textbf{Proof for Lemma 3.1.}

\begin{proof}
For a given allowed set $S$ of a weakly decomposable norm $p$,
denote
\begin{eqnarray*}
&&C_{1}=\Big\{z\in\R^{n}\Big| p(z_{S})\leq 1,\ z^{\bar{S}}=0\Big\},\quad
C_{2}=\Big\{z\in\R^{n}\Big| p^{\bar{S}}(z^{\bar{S}})\leq 1,\ z^{S}=0\Big\},\\
&&C=\Big\{z\in\R^{n}\Big| p(z_{S})+p^{\bar{S}}(z^{\bar{S}})\leq 1\Big\}.
\end{eqnarray*}
Then we have that
\begin{eqnarray}
\delta^{*}_{C_{1}}(\beta)&=&\max_{z\in\R^{n}}\Big\{\langle z,\beta\rangle\Big| p(z_{S})\leq 1, z^{\bar{S}}=0\Big\}=\max_{z\in\R^{n}}\Big\{\langle z,\beta_{S}\rangle\Big| p(z_{S})\leq 1, z^{\bar{S}}=0\Big\}\nonumber\\
&\leq&\max_{z\in\R^{n}}\Big\{\langle z,\beta_{S}\rangle\Big| p(z)\leq 1\Big\}=p_{*}(\beta_{S}),\label{eq:conv-delta-dual}\\
\delta^{*}_{C_{2}}(\beta)&=&\max_{z\in\R^{n}}\Big\{\langle z,\beta\rangle\Big| p^{\bar{S}}(z^{\bar{S}})\leq 1, z^{S}=0\Big\}=\max_{z\in\R^{n}}\Big\{\langle z^{\bar{S}},\beta^{\bar{S}}\rangle\Big| p^{\bar{S}}(z^{\bar{S}})\leq 1\Big\}
\nonumber\\
&=&p^{\bar{S}}_{*}(\beta^{\bar{S}}).
\end{eqnarray}
Furthermore, on one hand, for any $x\in C_{1}$, $y\in C_{2}$ and $0\leq t\leq 1$, it is easy to prove that $tx+(1-t)y\in C$. That is $\mbox{conv}(C_{1}\cup C_{2})\subseteq C$. On the other hand, for any $z\in C$ with $z^{S}=0$ or $z^{\bar{S}}=0$, it is clear that $z\in \mbox{conv}(C_{1}\cup C_{2})$; and for any $z\in C$ with $z^{S}\neq0$ and $z^{\bar{S}}\neq0$, we can find $x=\frac{z_{S}}{p(z_{S})}\in C_{1}$ and $y=\frac{z_{\bar{S}}}{1-p(z_{S})}\in C_{2}$ such that $z=p(z_{S})x+(1-p(z_{S}))y$. Therefore we
have shown that $C=\mbox{conv}(C_{1}\cup C_{2})$.

Due to Theorem 5.6 of \cite{Rockafellar96}, we can prove the following fact easily.
\begin{eqnarray}\label{eq:conv-delta-fun}
\mbox{conv}(\delta_{C_{1}},\delta_{C_{2}})(\beta)=\delta_{\mbox{conv}(C_{1}\cup C_{2})}(\beta),\quad \forall\, \beta\in\R^{n},
\end{eqnarray}
where $\mbox{conv}(\delta_{C_{1}},\delta_{C_{2}})$ denotes the
greatest convex function that is less than or equal to
$\delta_{C_1}$ and $\delta_{C_2}$ pointwise over the entire $\R^n$.
Based on the above basic results (\ref{eq:conv-delta-dual})-(\ref{eq:conv-delta-fun}), $C=\mbox{conv}(C_{1}\cup C_{2})$ and Theorem 16.5 of \cite{Rockafellar96}, we have that
\begin{eqnarray*}
p_{*}(\beta)&=&\max_{z\in\R^{n}}\Big\{\langle \beta,z\rangle\Big| p(z)\leq 1\Big\}\leq\max_{z\in\R^{n}}\Big\{\langle \beta,z\rangle\Big| p(z_{S})+p^{\bar{S}}(z^{\bar{S}})\leq 1\Big\}\\
&=&\delta^{*}_{C}(\beta)=\delta^{*}_{\mbox{conv}(C_{1}\cup C_{2})}(\beta)=\left(\mbox{conv}(\delta_{C_{1}},\delta_{C_{2}})\right)^{*}(\beta)=\max\Big\{\delta^{*}_{C_{1}}(\beta),\delta^{*}_{C_{2}}(\beta)\Big\}\\
&\leq&\max\Big\{p_{*}(\beta_{S}),p^{\bar{S}}_{*}(\beta^{\bar{S}})\Big\}.
\end{eqnarray*}
The desired results follow by taking $\beta=\truebeta$ and dividing both sides of the above inequality by $\|\varepsilon\|$ with $\beta=\varepsilon^{T}X$, respectively.
\end{proof}

\noindent\textbf{Proof for Lemma 3.3.}
\begin{proof}
Since $\hat{\beta}\in\mathop{\mbox{argmin}}\limits_{\beta\in\R^{n}}\left\{h(\beta;\tauone,\tautwo,0,0,b)\right\}$ and $\pone$ is a convex function, we have
\begin{eqnarray*}
-\frac{X^{T}(X\hat{\beta}-b)}{\|X\hat{\beta}-b\|}-\tauone\hat{\beta}-\tautwo X^{T}(X\hat{\beta}-b)\in\lambda\partial
\pone(\hat{\beta}).
\end{eqnarray*}
Hence
\begin{eqnarray}\label{eq:lemma3.2-proof}
\lambda \pone(\beta)\geq \lambda
\pone(\hat{\beta})+\left\langle\frac{X^{T}\hat{\varepsilon}}{\|\hat{\varepsilon}\|}-\tauone\hat{\beta}+\tautwo X^{T}\hat{\varepsilon},\beta-\hat{\beta}\right\rangle.
\end{eqnarray}
Let $\beta=\truebeta$. Then the inequality (\ref{eq:lemma3.2-proof}) can be rearranged to
\begin{eqnarray*}
\Big(\tautwo+\frac{1}{\|\hat{\varepsilon}\|}\Big)\hat{\varepsilon}^{T}X(\truebeta-\hat{\beta})&\leq&
\lambda \pone(\truebeta)-\lambda \pone(\hat{\beta})+\tauone\hat{\beta}^{T}(\truebeta-\hat{\beta})
\;\leq\;
\lambda \pone(\truebeta)+\tauone\hat{\beta}^{T}\truebeta\\
&\leq&\lambda
\pone(\truebeta)+\tauone\pstar(\truebeta)\pone(\hat{\beta}).
\end{eqnarray*}
Note that the last inequality is obtained by the definition of $\pstar$.
The desired result now follows readily.
\end{proof}

\noindent \textbf{Proof for Lemma 3.4.}
\begin{proof}
Since $\hat{\beta}\in\mathop{\mbox{argmin}}\limits_{\beta\in\R^{n}}\left\{h(\beta;\tauone,\tautwo,0,0,b)\right\}$, we have
$h(\hat{\beta};\tauone,\tautwo,0,0,b)\leq h(\truebeta;\tauone,\tautwo,0,0,b)$.
Thus, by the definition of the dual norm, we get
\begin{eqnarray}
\|\hat{\varepsilon}\|\leq
\|\varepsilon\|+\frac{\tautwo}{2}\|\varepsilon\|^{2}+\frac{\tauone}{2}\|\truebeta\|^{2}+\lambda
\pone(\truebeta)\leq\|\varepsilon\|+\frac{\tautwo}{2}\|\varepsilon\|^{2}+\left(\lambda+\frac{\tauone\pstar(\truebeta)}{2}\right)\pone(\truebeta),
\label{eq:lemma3.2-proof-2}\\
\lambda \pone(\hat{\beta})\leq\|\varepsilon\|+\frac{\tautwo}{2}\|\varepsilon\|^{2}+\frac{\tauone}{2}\|\truebeta\|^{2}+\lambda
\pone(\truebeta)\leq\|\varepsilon\|+\frac{\tautwo}{2}\|\varepsilon\|^{2}+\left(\lambda+\frac{\tauone\pstar(\truebeta)}{2}\right)\pone(\truebeta).\label{eq:lemma3.2-proof-3}
\end{eqnarray}
Dividing both sides of (\ref{eq:lemma3.2-proof-2}) by $\|\varepsilon\|$, we obtain
\begin{eqnarray*}
\frac{\|\hat{\varepsilon}\|}{\|\varepsilon\|}\leq
1+\frac{\tautwo}{2}\|\varepsilon\|+n_{p}+\frac{\tauone\pstar(\truebeta)p(\truebeta)}{2\|\varepsilon\|}=c_{u},
\end{eqnarray*}
where $c_u$ is defined in \eqref{eq:notation-a}.
In order to obtain the lower bound of
$\frac{\|\hat{\varepsilon}\|}{\|\varepsilon\|}$, we first
use the triangle inequality
$\|\hat{\varepsilon}\|=\|\varepsilon-X(\hat{\beta}-\truebeta)\|\geq
\|\varepsilon\|-\|X(\hat{\beta}-\truebeta)\|$,
and then the upper bound of $\|X(\hat{\beta}-\truebeta)\|$. By Lemma
3.3 and the definition of the dual norm, we have
\begin{eqnarray*}
&& \hspace{-0.7cm} \|X(\hat{\beta}-\truebeta)\|^{2}
\;=\; \varepsilon^{T}X(\hat{\beta}-\truebeta)+\hat{\varepsilon}^{T}X(\truebeta-\hat{\beta})\\
&\leq&\varepsilon^{T}X(\hat{\beta}-\truebeta)
+\kappa\left(\lambda \pone(\truebeta)+\tauone\pstar(\truebeta)\pone(\hat{\beta})\right)\\
&\leq&\lambda_{0}\pone(\hat{\beta}-\truebeta)\|\varepsilon\|+\kappa\left(\lambda \pone(\truebeta)+\tauone\pstar(\truebeta)\pone(\hat{\beta})\right)\\
&\leq&\lambda_{0}\left(\pone(\hat{\beta})+\pone(\truebeta)\right)\|\varepsilon\|+\kappa\left(\lambda \pone(\truebeta)+\tauone\pstar(\truebeta)\pone(\hat{\beta})\right)\\
&=&\left(\lambda_{0}\|\varepsilon\|+\lambda\kappa\right)\pone(\truebeta)
+\left(\lambda_{0}\|\varepsilon\|+\tauone\kappa\pstar(\truebeta)\right)\pone(\hat{\beta}),
\end{eqnarray*}
where $\kappa = \left(\tautwo+\frac{1}{\|\hat{\varepsilon}\|}\right)^{-1}.$
Substituting the inequality (\ref{eq:lemma3.2-proof-3}) into the above formula, we can obtain
\begin{eqnarray*}
\|X(\hat{\beta}-\truebeta)\|^{2}&\leq&\frac{\|\varepsilon\|+\frac{\tautwo}{2}\|\varepsilon\|^{2}}{\lambda}\left(\lambda_{0}\|\varepsilon\|+\tauone\kappa\pstar(\truebeta)
\right)\\
&&+\Big(2\lambda_{0}\|\varepsilon\|+(\lambda\kappa+\tauone\kappa\pstar(\truebeta))
+\frac{\tauone\pstar(\truebeta)}{2\lambda}(\lambda_{0}\|\varepsilon\|+\tauone\kappa\pstar(\truebeta))\Big)\pone(\truebeta).
\end{eqnarray*}
Rearranging the above inequality, we have
\begin{eqnarray*}
\|X(\hat{\beta}-\truebeta)\|&\leq&\|\varepsilon\|\sqrt{\hat{a}+\frac{2\lambda_{0}\pone(\truebeta)}{\|\varepsilon\|}+\frac{\|\hat{\varepsilon}\|}{\|\varepsilon\|}\frac{\left(\lambda+\tauone\pstar(\truebeta)\right) \pone(\truebeta)}{(1+\tautwo\|\hat{\varepsilon}\|)\|\varepsilon\|}}\\
&\leq&\|\varepsilon\|\sqrt{\hat{a}+\frac{2\lambda_{0}n_{\pone}}{\lambda}+\frac{\|\hat{\varepsilon}\|}{\|\varepsilon\|}\left(1+\frac{\tauone\pstar(\truebeta)}{\lambda}\right)n_{\pone}},
\end{eqnarray*}
where
\begin{eqnarray}\label{eq:def-a}
\hat{a}&=&\left(\frac{\left(1+\frac{\tautwo}{2}\|\varepsilon\|\right)}{\lambda}+\frac{\tauone\pstar(\truebeta)\pone(\truebeta)}{2\lambda\|\varepsilon\|}\right)\left(\lambda_{0}+\frac{\tauone\kappa\pstar(\truebeta)}{\|\varepsilon\|}\right)\nonumber\\
&\leq&\left(\frac{\left(1+\frac{\tautwo}{2}\|\varepsilon\|\right)}{\lambda}+\frac{\tauone\pstar(\truebeta)\pone(\truebeta)}{2\lambda\|\varepsilon\|}\right)\left(\lambda_{0}+\tauone\pstar(\truebeta)c_{u}\right)=a.
\end{eqnarray}
Therefore, by noting that $\norm{X(\hat{\beta}-\truebeta)} = \norm{\varepsilon-
\hat{\varepsilon}}$ and triangle inequality, we have
\begin{eqnarray*}
\frac{\|\hat{\varepsilon}\|}{\|\varepsilon\|}\geq
1-\sqrt{a+\frac{2\lambda_{0}n_{\pone}}{\lambda}+\frac{\|\hat{\varepsilon}\|}{\|\varepsilon\|}\left(1+\frac{\tauone\pstar(\truebeta)}{\lambda}\right)n_{\pone}}.
\end{eqnarray*}
By rearranging the above inequality, in the case when
$\frac{\|\hat{\varepsilon}\|}{\|\varepsilon\|}<1$, we
further derive that
\begin{eqnarray*}
a+\frac{2\lambda_{0}n_{\pone}}{\lambda}+\frac{\|\hat{\varepsilon}\|}{\|\varepsilon\|}\left(1+\frac{\tauone\pstar(\truebeta)}{\lambda}\right)n_{\pone}\geq\left(1-\frac{\|\hat{\varepsilon}\|}{\|\varepsilon\|}\right)^{2}\geq 1- \frac{2\|\hat{\varepsilon}\|}{\|\varepsilon\|}.
\end{eqnarray*}
Then we can obtain that
\begin{eqnarray*}
\frac{\|\hat{\varepsilon}\|}{\|\varepsilon\|}\geq\frac{1-a-\frac{2\lambda_{0}n_{\pone}}{\lambda}}{2+\left(1+\frac{\tauone\pstar(\truebeta)}{\lambda}\right)n_{\pone}}:=c_{l}>0.
\end{eqnarray*}
In the other case, if
$\frac{\|\hat{\varepsilon}\|}{\|\varepsilon\|}\geq 1$, we
have already obtain a lower bound that is larger than $c_{l}$.
\end{proof}

\noindent\textbf{Proof for Theorem 3.1.}
\begin{proof}
First we note that if the following inequality holds
\begin{eqnarray*}
\langle X(\hat{\beta}-\truebeta),X(\hat{\beta}-\beta)\rangle&\leq&
-\delta\left((\tilde{\lambda}+\tilde{\lambda}_{m})\pone(\hat{\beta}_{S}-\beta)+(\hat{\lambda}-\tilde{\lambda}_{m})\pone^{\Sc}(\hat{\beta}^{\Sc})\right)\|\varepsilon\|
\\
&&+\tauone c_{u}\|\hat{\beta}-\truebeta\|\|\beta-\truebeta\|\|\varepsilon\|,
\end{eqnarray*}
then we can verify that the theorem is valid by the
following simple calculations:
\begin{eqnarray*}
\lefteqn{\ \ \|X(\hat{\beta}-\truebeta)\|^{2}-\|X(\beta-\truebeta)\|^{2}+2\delta\left((\tilde{\lambda}+\tilde{\lambda}_{m})\pone(\hat{\beta}_{S}-\beta)+(\hat{\lambda}-\tilde{\lambda}_{m})\pone^{\Sc}(\hat{\beta}^{\Sc})\right)\|\varepsilon\|}\\
&=&2\langle
X(\hat{\beta}-\truebeta),X(\hat{\beta}-\beta)\rangle-\|X(\beta-\hat{\beta})\|^{2}+2\delta\left((\tilde{\lambda}+\tilde{\lambda}_{m})\pone(\hat{\beta}_{S}-\beta)+(\hat{\lambda}-\tilde{\lambda}_{m})\pone^{\Sc}(\hat{\beta}^{\Sc})\right)\|\varepsilon\|\\
&\leq&-\|X(\beta-\hat{\beta})\|^{2}+2\tauone c_{u}\|\hat{\beta}-\truebeta\|\|\beta-\truebeta\|\|\varepsilon\|\\
&\leq& 2\tauone c_{u}\|\hat{\beta}-\truebeta\|\|\beta-\truebeta\|\|\varepsilon\|.
\end{eqnarray*}
Thus it is sufficient to show that the result is true if
\begin{eqnarray}\label{eq:thm-1-proof-assume}
\langle X(\hat{\beta}-\truebeta),X(\hat{\beta}-\beta)\rangle&\geq&
-\delta\left((\tilde{\lambda}+\tilde{\lambda}_{m})\pone(\hat{\beta}_{S}-\beta)+(\hat{\lambda}-\tilde{\lambda}_{m})\pone^{\Sc}(\hat{\beta}^{\Sc})\right)\|\varepsilon\|
\\
&&+\tauone c_{u}\|\hat{\beta}-\truebeta\|\|\beta-\truebeta\|\|\varepsilon\|.\nonumber
\end{eqnarray}
By the inequality (\ref{eq:lemma3.2-proof}) and the fact that
$\hat{\varepsilon}=X(\truebeta-\hat{\beta}+\varepsilon)$, we can get
\begin{eqnarray}\label{eq:proof-thm-2}
 \langle
X(\hat{\beta}-\truebeta),X(\hat{\beta}-\beta)\rangle+\lambda\kappa\pone(\hat{\beta})
&\leq &
\langle\varepsilon,X(\hat{\beta}-\beta)\rangle
+\tauone\kappa\langle\hat{\beta},\beta-\hat{\beta}\rangle
+\lambda\kappa\pone(\beta),
\end{eqnarray}
where $\kappa := \left(\tautwo+\frac{1}{\|\hat{\varepsilon}\|}\right)^{-1}.$
Since $\mbox{supp}(\beta)\subseteq S$, it follows from the definition of the dual norm and the generalized Cauchy-Schwartz inequality that
\begin{eqnarray}\label{eq:proof-thm-3}
&& \hspace{-0.7cm}
\langle\varepsilon,X(\hat{\beta}-\beta)\rangle
\;=\;\langle\varepsilon,X(\hat{\beta}_{S}-\beta)\rangle
+ \langle\varepsilon,X(\hat{\beta}_{\Sc}-\beta)\rangle
\nonumber\\
&\leq&
\left(\pone_{*}((\varepsilon^{T}X)_{S})\pone(\hat{\beta}_{S}-\beta)+\pone^{\Sc}_{*}((\varepsilon^{T}X)^{\Sc})\pone^{\Sc}(\hat{\beta}^{\Sc})\right)
\;\leq\;
\lambda_{m}\left(\pone(\hat{\beta}_{S}-\beta)+\pone^{\Sc}(\hat{\beta}^{\Sc})\right)\|\varepsilon\|.
\end{eqnarray}
By substituting (\ref{eq:proof-thm-3}) into (\ref{eq:proof-thm-2}), we obtain
\begin{eqnarray}\label{eq:proof-thm-4}
\lefteqn{\ \ \langle
X(\hat{\beta}-\truebeta),X(\hat{\beta}-\beta)\rangle+\lambda\kappa\pone(\hat{\beta})}\nonumber\\
&\leq&\lambda_{m}\left(\pone(\hat{\beta}_{S}-\beta)
+\pone^{\Sc}(\hat{\beta}^{\Sc})\right)\|\varepsilon\|
+\tauone\kappa\langle\hat{\beta},\beta-\hat{\beta}\rangle
+\lambda\kappa\pone(\beta).
\end{eqnarray}
Furthermore, by using the weak decomposability and the triangle inequality in (\ref{eq:proof-thm-4}) we derive
\begin{eqnarray}\label{eq:proof-thm-5}
\lefteqn{\ \ \langle
X(\hat{\beta}-\truebeta),X(\hat{\beta}-\beta)\rangle
+\lambda\kappa\left(\pone^{\Sc}(\hat{\beta}^{\Sc})+\pone(\hat{\beta}_{S})\right)}\nonumber\\
&\leq&\lambda_{m}\left(\pone(\hat{\beta}_{S}-\beta)
+\pone^{\Sc}(\hat{\beta}^{\Sc})\right)\|\varepsilon\|
+\tauone\kappa\langle\hat{\beta},\beta-\hat{\beta}\rangle
+\lambda\kappa\left(\pone(\hat{\beta}_{S})+\pone(\hat{\beta}_{S}-\beta)\right).
\end{eqnarray}
Then by eliminating $\lambda\kappa\pone(\hat{\beta}_{S})$ on both sides of (\ref{eq:proof-thm-5}) and using the weak decomposability, we get
\begin{eqnarray}\label{eq:proof-thm-6}
\lefteqn{\ \ \langle
X(\hat{\beta}-\truebeta),X(\hat{\beta}-\beta)\rangle
+\lambda\kappa\pone^{\Sc}(\hat{\beta}^{\Sc})}\nonumber\\
&\leq&\lambda_{m}\left(\pone(\hat{\beta}_{S}-\beta)
+\pone^{\Sc}(\hat{\beta}^{\Sc})\right)\|\varepsilon\|
+\tauone\kappa\langle\hat{\beta},\beta-\hat{\beta}\rangle
+\lambda\kappa \pone(\hat{\beta}_{S}-\beta)
\nonumber\\
&\leq&\lambda_{m}\left(\pone(\hat{\beta}_{S}-\beta)
+\pone^{\Sc}(\hat{\beta}^{\Sc})\right)\|\varepsilon\|
+\tauone\kappa\langle\truebeta,\beta-\hat{\beta}\rangle
+\tauone\kappa\langle\hat{\beta}-\truebeta,\beta-\truebeta\rangle
+\lambda\kappa\pone(\hat{\beta}_{S}-\beta)
\nonumber\\
&\leq&\lambda_{m}\left(\pone(\hat{\beta}_{S}-\beta)+\pone^{\Sc}(\hat{\beta}^{\Sc})\right)\|\varepsilon\|+\lambda\kappa\pone(\hat{\beta}_{S}-\beta)
+\lambda_{m}\tauone\kappa\left(\pone(\hat{\beta}_{S}-\beta)+\pone^{\Sc}(\hat{\beta}^{\Sc})\right)
\\
&& +\tauone\kappa\|\hat{\beta}-\truebeta\|\|\beta-\truebeta\|. \nonumber
\end{eqnarray}
By using the result of Lemma 3.4, the inequality (\ref{eq:proof-thm-6}) becomes
\begin{eqnarray}\label{eq:proof-thm-7}
\lefteqn{\ \ \langle
X(\hat{\beta}-\truebeta),X(\hat{\beta}-\beta)\rangle+\left(\hat{\lambda}-\tilde{\lambda}_{m}\right)\pone^{\Sc}(\hat{\beta}^{\Sc})\|\varepsilon\|}\nonumber\\
&\leq&\left(\tilde{\lambda}+\tilde{\lambda}_{m}\right)\pone(\hat{\beta}_{S}-\beta)\|\varepsilon\|
+\tauone c_{u}\|\hat{\beta}-\truebeta\|\|\beta-\truebeta\|\|\varepsilon\|.
\end{eqnarray}
From the condition (\ref{eq:thm-1-proof-assume}) in (\ref{eq:proof-thm-7}) and
simple rearrangement, we have that
\begin{eqnarray*}
\left(\hat{\lambda}-\tilde{\lambda}_{m}\right)(1-\delta)\pone^{\Sc}(\hat{\beta}^{\Sc})
\leq\left(\tilde{\lambda}+\tilde{\lambda}_{m}\right)(1+\delta)\pone(\hat{\beta}_{S}-\beta).
\end{eqnarray*}
By Lemma 3.1 we have $\lambda_{0}\leq\lambda_{m}$ and $p_{*}(\truebeta)\leq\lambda_{m}$.
Since
\begin{eqnarray*}
\frac{\lambda-\lambda_{m}t_{1}(1+\sigma t_{1}+\sigma n_{p})-2\lambda_{m}n_{p}}{\lambda\left(2+n_{p}+\frac{\sigma p_{*}(\truebeta)p(\truebeta)}{\|\varepsilon\|}\right)}<c_{l}<\frac{1}{2+n_{p}+\frac{\sigma p_{*}(\truebeta)p(\truebeta)}{\|\varepsilon\|}},
\end{eqnarray*}
we can see that
\begin{eqnarray*}
\hat{\lambda}>\frac{\lambda-\lambda_{m} t_{1}(1+\sigma t_{1}+\sigma n_{p})-2\lambda_{m}n_{p}}{t_{2}+n_{p}}.
\end{eqnarray*}
Then it is easy to find that if
\begin{eqnarray}\label{eq:def-c}
\frac{s_{2}-\sqrt{s_{2}^{2}-4s_{1}s_{3}}}{2s_{1}}<\lambda<\frac{s_{2}+\sqrt{s_{2}^{2}-4s_{1}s_{3}}}{2s_{1}},
\end{eqnarray}
then we have $\hat{\lambda} = \lambda c_l/(1+\tautwo c_l) >\tilde{\lambda}_{m} = \lambda_m (1+\tauone c_u)$.

Furthermore,
\begin{eqnarray*}
\pone^{\Sc}(\hat{\beta}^{\Sc})\leq\left(\frac{\tilde{\lambda}+\tilde{\lambda}_{m}}{\hat{\lambda}-\tilde{\lambda}_{m}}\right)\cdot\frac{1+\delta}{1-\delta}\cdot
\pone(\hat{\beta}_{S}-\beta).
\end{eqnarray*}
From the definition of $L_{S}$ and Lemma 3.2 with the assumption $\mbox{supp}(\beta)\subseteq S$, it follows that
\begin{eqnarray*}
&&\pone^{\Sc}(\hat{\beta}^{\Sc})\leq
L_{S}\pone(\hat{\beta}_{S}-\beta),\quad \pone(\hat{\beta}_{S}-\beta)\leq\Gamma_{\pone}(L_{S},S)\|X(\hat{\beta}-\beta)\|.
\end{eqnarray*}
By using the inequality (\ref{eq:proof-thm-7}), we can derive that
\begin{eqnarray*}
&&\langle X(\hat{\beta}-\truebeta),X(\hat{\beta}-\beta)\rangle +
\delta\|\varepsilon\|(\hat{\lambda}-\tilde{\lambda}_{m})\pone^{\Sc}(\hat{\beta}^{\Sc})
\\[5pt]
&&\leq(\tilde{\lambda}+\tilde{\lambda}_{m})\pone(\hat{\beta}_{S}-\beta)\|\varepsilon\|+\tauone c_{u}\|\hat{\beta}-\truebeta\|\|\beta-\truebeta\|\|\varepsilon\|
\\[5pt]
&&\leq(1+\delta)(\tilde{\lambda}+\tilde{\lambda}_{m})\Gamma_{\pone}(L_{S},S)\|X(\hat{\beta}-\beta)\|\|\varepsilon\|-\delta(\tilde{\lambda}+\tilde{\lambda}_{m})\pone(\hat{\beta}_{S}-\beta)\|\varepsilon\|\\
&&\quad+\tauone c_{u}\|\hat{\beta}-\truebeta\|\|\beta-\truebeta\|\|\varepsilon\|.
\end{eqnarray*}
Noticing that
\begin{eqnarray*}
& 2\langle
X(\hat{\beta}-\truebeta),X(\hat{\beta}-\beta)\rangle=\|X(\hat{\beta}-\truebeta)\|^{2}-\|X(\beta-\truebeta)\|^{2}+\|X(\hat{\beta}-\beta)\|^{2}, & \\
& 2(1+\delta)|(\tilde{\lambda}+\tilde{\lambda}_{m})\Gamma_{\pone}(L_{S},S)\|X(\hat{\beta}-\beta)\|\|\varepsilon\|\leq\left((1+\delta)(\tilde{\lambda}+\tilde{\lambda}_{m})\Gamma_{\pone}(L_{S},S)\right)^{2}\|\varepsilon\|
+\|X(\hat{\beta}-\beta)\|^{2}, &
\end{eqnarray*}
we get
\begin{eqnarray*}
&&\quad\|X(\hat{\beta}-\truebeta)\|^{2}+2\delta\left((\hat{\lambda}-\tilde{\lambda}_{m})\pone^{\Sc}(\hat{\beta}^{\Sc})+(\tilde{\lambda}+\tilde{\lambda}_{m})\pone(\hat{\beta}_{S}-\beta)\right)\|\varepsilon\|\\
&&\quad\leq\|X(\beta-\truebeta)\|^{2}+(1+\delta)^{2}(\tilde{\lambda}+\tilde{\lambda}_{m})^{2}\Gamma_{\pone}^{2}(L_{S},S)\|\varepsilon\|^{2}+2\tauone c_{u}\|\hat{\beta}-\truebeta\|\|\beta-\truebeta\|\|\varepsilon\|.
\end{eqnarray*}
Therefore the oracle inequality holds and this completes the proof.
\end{proof}

\section*{Acknowledgements}
We would like to thank Prof. Jian Huang  and Dr. Ying Cui   for bringing the references \cite{SunZ2012} and \cite{XuCM}, respectively, to our attention. We also thank Prof. Sara van de Geer at ETH Z$\ddot{u}$rich and Dr. Benjamin Stucky at University of Z$\ddot{u}$rich for the fruitful discussions. Last but not least, we thank the action editor Dr. David Wipf and the anonymous referees for their helpful suggestions to improve the manuscript.


\end{document}